\documentclass[a4paper,12pt]{amsart}
\usepackage[utf8]{inputenc}

\usepackage{url}
\usepackage[margin=1in]{geometry}  
\usepackage{epsfig}
\usepackage{color}
\usepackage{graphicx}              
\usepackage{amsmath}
\usepackage{amscd}               
\usepackage{amsfonts}
\usepackage{amssymb}             
\usepackage{amsthm}                
\usepackage{epsfig}
\usepackage{mathrsfs}
\usepackage{hyperref}
\usepackage{enumerate}
\usepackage{esint}

\newtheorem{theorem}{Theorem}

\newtheorem*{theorem*}{Theorem}
\newtheorem{lemma}[theorem]{Lemma}
\newtheorem*{lemma*}{Lemma}
\newtheorem{proposition}[theorem]{Proposition}
\newtheorem*{proposition*}{Proposition}
\theoremstyle{definition}

\theoremstyle{remark}

\newcommand{\tr}{\operatorname{tr}}
\newcommand{\Disc}{\operatorname{Disc}}
\newcommand{\RE}{\operatorname{Re}}
\newcommand{\IM}{\operatorname{Im}}

\newcommand{\GCD}{\operatorname{GCD}}

\newcommand{\sgn}{\operatorname{sgn}}

\newcommand{\sym}{\operatorname{sym}}

\newcommand{\E}{\mathbf{E}}     
\newcommand{\one}{\mathbf{1}}

\newcommand{\bC}{\mathbb{C}}

\newcommand{\bQ}{\mathbb{Q}}
\newcommand{\bR}{\mathbb{R}}

\newcommand{\zed}{\mathbb{Z}}

\newcommand{\GL}{\mathrm{GL}}

\newcommand{\SL}{\mathrm{SL}}

\newcommand{\sE}{{\mathscr{E}}}

\newcommand{\sI}{{\mathscr{I}}}

\newcommand{\sL}{{\mathscr{L}}}
\newcommand{\sM}{{\mathscr{M}}}
\newcommand{\sN}{{\mathscr{N}}}

\newcommand{\sR}{{\mathscr{R}}}
\newcommand{\sS}{{\mathscr{S}}}

\newcommand{\fS}{\mathfrak{S}}

\title{Lower order terms in the shape of cubic fields}
\author{Robert D. Hough}
\address{Department of Mathematics, Stony Brook University, 100 Nicolls Road, Stony Brook, NY 11794}
\email{robert.hough@stonybrook.edu}
\author{Eun Hye Lee}
\address{Department of Mathematics, Texas Christian University, Tucker Technology Center,
Suite 358,
2840 West Bowie Street,
Fort Worth, Texas 76109}
\email{eun.hye.lee@tcu.edu }
\subjclass[2010]{Primary 11M41, 11F68, 11H06, 11E45, 12F05, 43A85, 43A90}
 \keywords{Cubic ring,  equidistribution, Eisenstein series, space of lattices, zeta function, prehomogeneous vector space}
\begin{document}

\begin{abstract}
We demonstrate equidistribution of the lattice shape of cubic fields when ordered by discriminant, giving an estimate in the Eisenstein series spectrum with a lower order main term.  The analysis gives a separate discussion of the contributions of reducible and irreducible binary cubic forms, following a method of Shintani.  Our work answers a question posed at the American Institute of Math by giving a precise geometric and spectral description of an evident barrier to equidistribution in the lattice shape.
\end{abstract}

\thanks{This material is based upon work supported by the National Science
Foundation under agreement DMS-1802336. Any opinions, findings and
conclusions or recommendations expressed in this material are those of the
authors and do not necessarily reflect the views of the National Science
Foundation.}

\thanks{Robert Hough is supported by an Alfred P. Sloan Foundation Research 
Fellowship and a Stony Brook Trustees Faculty Award}
\maketitle

\section{Introduction}

The image of the ring of integers of a degree $n$ number field $K$ in the canonical embedding is a rank $n$ lattice of covolume $2^{-r_2} |D|^{\frac{1}{2}}$ where $r_2$ is the number of complex embeddings and $D$ is the field discriminant.
It was noticed at a meeting at the American Institute of Math that there is an irregularity in the distribution of lattice shapes that appear as the ring of integers of cubic number fields.  Indeed, since the norm of an algebraic integer is a non-zero integer, there is a constraint on the shortest vector in this lattice.  The purpose of this article is to study this phenomenon in greater detail by studying the distribution of the lattice shape in the Eisenstein spectrum of the space of lattices.  See \cite{H19} for the corresponding analysis in the cuspidal spectrum, where, note, the irregularity of distribution does not appear.

Let $\Lambda_K$ denote the lattice shape of the ring of integers of a number field $K$, which is defined to be the orthogonal projection of the ring of integers in the space orthogonal to $\sigma(1)$ in the trace paring, rescaled to have co-volume 1.  The lattice shape is identified with a point in $\SL_{n-1}(\zed)\backslash \SL_{n-1}(\bR)$ by a choice of base point. This shape is known to be asymptotically equidistributed in the case of $S_n$ number fields of degrees 3, 4 and 5 by work of Bhargava and Harron \cite{BH16}, which we make effective by giving a rate of convergence in the corresponding Weyl sums.  Let $\E_r$ denote the real analytic Eisenstein series for $\SL_2(\bR)$ with Laplace eigenvalue $\lambda = \frac{1}{4} + r^2$.
\begin{theorem}\label{main_theorem}
 Let $F$ be a smooth function of compact support on $\bR^+$. We have the following evaluation of Weyl sums
\begin{align*}
S^+(\mathbf E_r,F,X)=& \sum_{\substack{[K:\bQ] = 3, S_3}} F\left(\frac{\Disc K}{X}\right)\E_r(\Lambda_K) \\
={}&3^{-\frac{7+z}{4}}\tilde F\left(\frac{11+z}{12}\right) X^{\frac{11+z}{12}} 2^{\frac{z-7}{6}}\pi^{\frac{1+2z}{6}} \cos\left(\frac{\pi(1-z)}6\right) \frac{\zeta\left(\frac{1-z}{3}\right) \Gamma\left(\frac{1-z}{3}\right) \Gamma\left(\frac{2+z}{6}\right)}{\Gamma\left(\frac{5+z}{6}\right)\zeta(2)\zeta\left(\frac{7+2z}{3}\right)}\\
&+\frac{\xi(z)}{\xi(1+z)} 3^{-\frac{7-z}{4}} \tilde F\left(\frac{11-z}{12}\right) X^{\frac{11-z}{12}} 2^{\frac{-z-7}{6}}\pi^{\frac{1-2z}{6}} \cos\left(\frac{\pi(1+z)}6\right) \frac{\zeta\left(\frac{1+z}{3}\right)\Gamma\left(\frac{1+z}{3}\right)\Gamma\left(\frac{2-z}{6}\right)}{\Gamma\left(\frac{5-z}{6}\right)\zeta(2)\zeta\left(\frac{7-2z}{3}\right)}\\
&+O_{F,z,\epsilon}(X^{\frac{13}{15}+\epsilon}),\\
S^-(\mathbf E_r,F,X)={}& \sum_{\substack{[K:\bQ] = 3, S_3}} F\left(\frac{-\Disc K}{X}\right) \E_r(\Lambda_K) \\
&=\frac13 \tilde F\left(\frac{11+z}{12}\right) X^{\frac{11+z}{12}} 2^{\frac{z-7}{6}}\pi^{\frac{1+2z}{6}} \cos\left(\frac{\pi(1-z)}6\right) \frac{\zeta\left(\frac{1-z}{3}\right)\Gamma\left(\frac{1-z}{3}\right)\Gamma\left(\frac{2+z}{6}\right)}{\Gamma\left(\frac{5+z}{6}\right)\zeta(2)\zeta\left(\frac{7+2z}{3}\right)}\\
&+\frac{\xi(z)}{\xi(1+z)} \frac13 \tilde F\left(\frac{11-z}{12}\right)
X^{\frac{11-z}{12}} 2^{\frac{-z-7}{6}}\pi^{\frac{1-2z}{6}} \cos\left(\frac{\pi(1+z)}6\right)\frac{\zeta\left(\frac{1+z}{3}\right)\Gamma\left(\frac{1+z}{3}\right)\Gamma\left(\frac{2-z}{6}\right)}{\Gamma\left(\frac{5-z}{6}\right)\zeta(2)\zeta\left(\frac{7-2z}{3}\right)}\\
&+O_{F,z,\epsilon}(X^{\frac{13}{15}+\epsilon}).
\end{align*}
\end{theorem}
Note that the number of cubic fields with discriminant of size at most $X$ grows linearly in $X$, so that the above theorem obtains a power saving estimate in the Weyl sums.  Combined with the estimate of the cuspidal Weyl sums in \cite{H19} where we obtained a bound $O\left(X^{\frac{2}{3} + \epsilon}\right)$, this gives a quantitative proof of the asymptotic equidistribution in the cubic case.

The proof of Theorem \ref{main_theorem} is based on an analysis of the orbital twisted Shintani zeta functions, combined with a sieve.  This method of proving asymptotic results for cubic fields was introduced by Taniguchi and Thorne \cite{TT13a}, \cite{TT13b}, see also \cite{BST13}.  Here we extend the Eisenstein series twisted zeta functions introduced in \cite{HL20} to the orbital case, obtaining their poles and residues and matching off a pole arising from reducible binary cubics.  A detailed discussion of the poles and residues appears in Section \ref{twisted_zeta_function_section}.  We give a heuristic explanation of the existence of the poles and of the secondary main term and its size following our discussion of the parameterization in homogeneous coordinates of the space of binary cubic forms.

\subsection{Related work} The asymptotic count for the number of cubic fields of bounded discriminant was determined by Davenport and Heilbronn \cite{DH71}. Terr \cite{T97} first proved the asymptotic equidistribution of the shape of cubic fields.  Sato and Shintani \cite{SS74}, \cite{S72}, \cite{S75} introduced zeta functions enumerating integral orbits in prehomogeneous vector spaces, such as the space of binary cubic forms, which are in bijection with cubic rings up to isomorphism due to works of Delone and Fadeev \cite{DF64} and Gan-Gross-Savin \cite{GGS02}.  Orbital versions of these zeta functions were first used to prove the Davenport-Heilbronn Theorem by Taniguchi and Thorne \cite{TT13a}, \cite{TT13b}.  The first author introduced twisted versions of the zeta functions, twisting by an automorphic form evaluated on the lattice shape \cite{H17} and used these to prove the cuspidal equidistribution of the lattice shape of cubic fields in \cite{H19}.  The authors evaluate the poles and residues of the Eisenstein series twist in \cite{HL20}, and use these methods here to prove the equidistribution in the Eisenstein spectrum.  On the binary cubic form side, the Eisenstein spectrum studies roughly the joint distribution of the first and fourth covariants of a binary cubic form.  The second author has studied the joint equidistribution of the first and second co-variants in \cite{LTB22}.

\section*{Notation and conventions}

We abbreviate the contour integral \begin{equation}\frac{1}{2\pi i} \int_{c-i\infty}^{c + i \infty} \psi(z) dz = \oint_{\RE(z) = c} \psi(z)dz.\end{equation} Denote $e(x)=e^{2\pi i x}$, $c(x) = \cos(2\pi x)$, $s(x) = \sin(2\pi x)$.
The argument uses the following pair of standard Mellin transforms.  Write $K_\nu$ for the $K$-Bessel function.
For $\RE(s)> |\RE \nu|$, (\cite{I02}, p.205)
\begin{equation}
 \int_0^\infty K_\nu(x)x^{s-1}dx = 2^{s-2}\Gamma\left(\frac{s + \nu}{2}\right)\Gamma\left(\frac{s-\nu}{2}\right).
\end{equation}
We use the formula
\begin{equation}
 K_{\frac{s}{2}}(2) = \int_0^\infty e^{-t^2 -\frac{1}{t^2}} t^{s-1} dt.
\end{equation}
For $0 < \RE(s) < 1$, (\cite{B53}, p.13)
\begin{equation}
 \int_0^\infty \cos(x)x^{s-1}dx = \Gamma(s)\cos\left( \frac{\pi}{2}s\right).
\end{equation}
For functions $f$ on Euclidean space $\bR^d$, and $t \in \bR^\times$ we use the notation $f^t(x) = f(t\cdot x)$.  Under Fourier transform, this satisfies $\widehat{f^t}(\xi)= \frac{1}{|t|^d} \hat{f}\left(\frac{\xi}{t}\right)$.

The following  groups are used. 
\begin{itemize}
 \item $G_\bR = \GL_2(\bR)$
 \item $G^1 = \SL_2(\bR)$
 \item $G^+ = \{g \in \GL_2(\bR): \det g>0\}$
 \item $G_\zed = \GL_2(\zed)$
 \item $\Gamma = \SL_2(\zed)$
 \item $\Gamma_\infty = \left\{ \begin{pmatrix} \pm 1 &\\ n & \pm 1\end{pmatrix} : n \in \zed\right\}$
 \item $A = \{a_t: t \in \bR_{>0}\},$ $ a_t = \begin{pmatrix} t &\\ & \frac{1}{t}\end{pmatrix}$
 \item $N = \{n_x: x \in \bR\},$ $n_x = \begin{pmatrix} 1 &0 \\ x &1 \end{pmatrix}$
 \item $N' = \{\nu_x: x \in \bR\},$ $\nu_x =\begin{pmatrix} 1 &x\\0&1 \end{pmatrix}$
 \item $K = \{k_\theta: \theta \in \bR/ \zed\},$ $k_\theta = \begin{pmatrix} \cos 2\pi \theta & \sin 2\pi\theta \\ -\sin 2\pi\theta & \cos 2\pi\theta\end{pmatrix}$
 \item $B = \left\{\begin{pmatrix} b_{11} &0\\b_{21} & b_{22}\end{pmatrix}: b_{11}, b_{21}, b_{22} \in \bR, b_{11}b_{22} \neq 0 \right\}$, $B^+ = \{b \in B: b_{11}, b_{22}>0\}$
 \item $d_\lambda = \begin{pmatrix} \lambda &\\ & \lambda \end{pmatrix}$.
\end{itemize}

We follow Shintani's conventions \cite{S72} regarding integrals and automorphic forms on $\SL_2(\bR)$. The Iwasawa decomposition is $G = KAN$ with Haar measure, for $f \in L^1(G^1)$,
\begin{align}
 \int_{G^1} f(g) dg &=  \int_0^{1}\int_{-\infty}^\infty \int_0^\infty f(k_\theta a_t n_u)\frac{dt}{t^3}du d\theta
\end{align}
and for $f \in L^1(G^+)$,
\begin{equation}
 \int_{G^+}f(g)dg = \int_0^\infty \int_{G^1} f\left(\begin{pmatrix} \lambda &\\ &\lambda\end{pmatrix} g \right)dg \frac{d\lambda}{\lambda}.
\end{equation}
Given a group element $g$, write $k(g), t(g), u(g)$ for the elements of $K, A, N$ in the representation of $g$ in the Iwasawa decomposition.  
We use several times the Siegel set $\fS_C$ 
\begin{equation}
 \fS_C = \left\{k_\theta a_T n_u: T\ge C , |u| \leq \frac{1}{2}\right\}
\end{equation}
which satisfies $\fS_{\frac{1}{2}}\cdot \Gamma = G^1$.
For any $r \in \bR$, define the semi-norm 
\begin{equation}
 \mu(r)(f) = \sup_{g \in \fS_{\frac{1}{2}}} t(g)^r |f(g)|.
\end{equation}
Let $C(G^1/\Gamma, r) = \{f \in C(G^1/\Gamma), \mu(r)(f)<\infty\}$.  

\section{Review of Eisenstein series}

Shintani's normalization of the Eisenstein series makes this right $\Gamma$ and left $K$ invariant, 
\begin{equation}\label{def_real_analytic_eisenstein}
E(z,g) = \sum_{\gamma \in \Gamma/\Gamma_\infty}
t(g \gamma)^{z+1}.
\end{equation}
 The function $E(z,g)$ satisfies the functional equation
\begin{equation}
\xi(z + 1)E(z,g) = \xi(1-z)E(-z,g); \qquad \xi(z) =
\pi^{-\frac{z}{2}}\Gamma\left(\frac{z}{2}\right)\zeta(z)
\end{equation}
and has a Fourier development in $z \neq 0$ given by
\begin{align}\label{eisenstein_fourier_dev}
E(z,g)&= t^{1+z}+ t^{1-z}\frac{\xi(z)}{\xi(z+1)}\\
\notag & \qquad+\frac{4t}{\xi(z+1)} \sum_{m=1}^\infty\eta_{\frac{z}{2}}(m)K_{\frac{z}{2}}(2\pi m t^2)\cos 2\pi m u,\\
\notag\eta_{\frac{z}{2}}(m)&=\sum_{ab = m}\left(\frac{a}{b} \right)^{\frac{z}{2}},
\end{align}
with $K_\nu$ the $K$ Bessel function.  We use frequently that $E(z, g) = E(z, (g^{-1})^t)$. The Riemann $\xi$ function $\xi(z)$ satisfies the functional equation $\xi(z) = \xi(1-z)$.

We also use the incomplete Eisenstein series to regularize integrals in the same way as Shintani. Let $\Psi$ denote the space of entire
functions such that for all $\psi \in \Psi$, for all $-\infty < C_1 < C_2 <
\infty$, for all $ N >0$,
\begin{equation}
 \sup_{C_1 < \RE (w) < C_2}\left(1 + (\IM w)^2\right)^N |\psi(w)| < \infty.
\end{equation}
For $\psi \in \Psi$ and $\RE(w) > 1$ choose $1 < c < \RE (w)$ and set
\begin{equation}
\sE(\psi, w; g) = \oint_{\RE (z) = c}\psi(z)\frac{E(z,g)}{w-z}dz.
\end{equation}
Shintani Lemma 2.9 gives the following estimates.
\begin{lemma}
 We have
 \begin{enumerate}
  \item $\sE(\psi, w; g) \in C(G^1/\Gamma, \RE w-1)$
  \item For a fixed $\psi$,
  \begin{equation}
   \sup_{1 \leq w \leq M, g \in \fS_{\frac{1}{2}}} |(w-1)\sE(\psi, w;g)| < \infty, (M>1)
  \end{equation}
  \item $\lim_{w \to 1^+} (w-1)\sE(\psi, w;g) = \frac{\psi(1)}{\xi(2)}.$
 \end{enumerate}

\end{lemma}

The corollary to Lemma 2.9 in \cite{S72} states that, for $f \in L^1(G^1/\Gamma, dg)$,
\begin{equation}
 \lim_{w \to 1^+} (w-1) \int_{G^1/\Gamma} f(g) \sE(\psi, w;g)dg = \frac{\psi(1)}{\xi(2)} \int_{G^1/\Gamma} f(g)dg.
\end{equation}
Similarly, 
\begin{lemma}
For $0 < c < w$ and $f \in L^1(G^1/\Gamma_\infty)$,
\begin{equation}
 \lim_{w \to 0^+} w \int_{G^1/\Gamma_\infty} f(g) \oint_{\RE(z) = c}\frac{\psi(\alpha)t(g)^\alpha}{\alpha (w-\alpha)} d\alpha dg= \psi(0)\int_{G^1/\Gamma_\infty} f(g)dg.
\end{equation}
 
\end{lemma}
\begin{proof}
Let $w< \frac{1}{2}$,  $F_w(t) = \oint_{\RE(z) = c}\frac{\psi(\alpha)t^\alpha}{\alpha (w-\alpha)} d\alpha$.
 Let $\epsilon > 0$.  For $t > \epsilon$, shift the contour left to the line $\RE(\alpha) = -1$, where the integral is uniformly bounded in $w$ and $\epsilon$.  A pole is passed at 0 with residue $\frac{\psi(0)}{w}$.  If $t \leq \epsilon$, shift the contour right to $\RE(\alpha) = 1$, passing a pole at $w$ with residue $-\frac{\psi(w)t^w}{w}$ with an integral that is uniformly bounded in $w$.  Letting $w \to 0$ obtains the claim.
 
\end{proof}

Let $f \in \sS(G^1)$ (Schwarz class) be bi-$K$-invariant, that is, for any $g \in G^1$ and $k_{\theta_1}, k_{\theta_2}$, $f(g) = f(k_{\theta_1}gk_{\theta_2})$. 
For imaginary $z = i \gamma$, $E(i\gamma, g) = \E_{\frac{\gamma}{2}}(g^{t})$, which is left invariant under $\Gamma$. Let $\E_r^c$ be the constant term in the Fourier expansion, and $\E_r^n = \E_r - \E_r^c$.
As a right convolution operator $f$ acts on $\E_r$ as multiplication by a scalar.  To check this, note that
\begin{equation}
 \E_r*f (g_0) = \int_{gh = g_0} \E_r(g)f(h)dh
\end{equation}
is left $\Gamma$ invariant and right $K$ invariant.  Also, it is an eigenfunction of the Laplacian and Hecke operators, with the same eigenvalues as $\E_r$.  It follows by multiplicity one that the convolution is a multiple of $\E_r$, $\E_r * f(g_0) = \Lambda_{f,r} \E_r(g_0)$, where $\Lambda_{f,r}$ is a scalar depending on $f$ and $r$. The following lemma, adapted from \cite[Section 1.8]{I02}, determines the eigenvalue.

\begin{lemma}
 We have$$\E_r * f = \left( \int_{G^1} f(g) t(g)^{1+z}dg\right) \E_r. $$
For the choice $f(g) = \exp\left(-\tr g^t g\right)$ the eigenvalue is $\sqrt{\pi}K_{\frac{z}{2}}(2)$.
\end{lemma}

\begin{proof}
 Let $\psi \in C_c^\infty(\Gamma \backslash G^1/K)$ be a smooth test function and let $\psi_0$ be the constant term in its Fourier expansion in the parabolic direction.  The Petersson inner product of $\E_r$ with $\psi$ is a Mellin transform of $\psi_0$,
 \begin{align*}
  \int_{\Gamma \backslash G^1} \E_r(g) \overline{\psi(g)} dg &= \int_{\Gamma \backslash G^1} \sum_{\gamma \in \Gamma_\infty \backslash \Gamma} t(\gamma g)^{1+z} \overline{\psi(\gamma g)} dg\\
  &= \int_{\Gamma_\infty \backslash G^1} t(g)^{1 + z} \overline{\psi(g)} dg\\
  &= \int_0^\infty \psi_0(t) t^{-1 + z} \frac{dt}{t} = \tilde{\psi}_0(-1+z).
 \end{align*}
Next we calculate the inner product with the convolution $\E_r * f$,
\begin{align*} 
 \int_{\Gamma \backslash G^1} (\E_r * f)(g) \overline{\psi(g)} dg &= \int_{\Gamma \backslash G^1} \int_{G^1} \E_r(h) f(h^{-1}g) \overline{\psi(g)} dh dg\\
 &= \int_{\Gamma \backslash G^1} \int_{G^1} \sum_{\gamma \in \Gamma_\infty \backslash \Gamma} t(h)^{1+z} f(h^{-1}\gamma g) \overline{\psi(\gamma g)} dh dg\\
 &= \int_{\Gamma_\infty \backslash G^1} \int_{G^1} t(h)^{1+z} f(h^{-1}g) \overline{\psi(g)}dh dg\\
 &= \int_0^\infty \frac{dt_1}{t_1^3} \int_0^\infty \frac{dt_2}{t_2^3} \int_{-\infty}^\infty du t_1^{1+z} f\left(\begin{pmatrix} \frac{t_2}{t_1} & \frac{u}{t_1t_2}\\ 0 & \frac{t_1}{t_2} \end{pmatrix}\right) \psi_0(t_2).
\end{align*}
After a change of coordinates we obtain
$$
 \tilde{\psi}_0(-1+z) \int_{0}^\infty \frac{dt}{t} t^{z} \int_{-\infty}^\infty du f\left(\begin{pmatrix} \frac{1}{t} & u \\ 0 & t\end{pmatrix}\right) = \tilde{\psi}_0(-1+z) \int_{0}^\infty \frac{dt}{t} t^{1+z} \int_{-\infty}^\infty du f\left(n_u a_t\right).
$$
Let $f(g) = \exp(-\tr(g^tg))$ so that the eigenvalue may be written
$$
 \int_{0}^\infty \frac{dt}{t} t^{z} \int_{-\infty}^\infty du \exp\left(-t^2 - \frac{1}{t^2} - u^2\right).
$$
Choose $H(z) = e^{-z}$ to find that the integral is $\sqrt{\pi}K_{\frac{z}{2}}(2)$.
\end{proof}

Because we have kept Shintani's convention regarding the action on forms, which uses the $KAN$ decomposition, but also use the modern convention that automorphic forms are left $\Gamma$ and right $K$ invariant, we often find it more convenient to write the convolution with the Eisenstein series in the equivalent form
$$
 \E_r *f(g_0) = \E_r \odot f(g_0) = \int_{h}\E_r(h^t)f(hg_0)dh
$$
which uses $\E_r(g) = \E_r((g^{-1})^t).$ Call $\E_r \odot f$ the \emph{pairing} of $\E_r$ and $f$.

The following lemma controls the contribution of the constant and  non-constant terms of the Eisenstein series to the pairing.  
\begin{lemma}\label{non_constant_term_bound_convolution}
For $((g_0)^{-1})^t = \nu_{x_0}a_{t_0}k_{\theta_0}$ varying in the Siegel set $\{t_0 \geq \frac{1}{2}, |x_0| \leq \frac{1}{2}\}$, the pairing of the non-constant terms $\E_r^n \odot f(g_0) = \int_h \E_r^n(h^t)f(hg_0)= O_{f,A}(t_0^{-A})$.  For $((g_0)^{-1})^t$ varying in the complement of a Siegel set $\{t_0 \leq 10, |x_0| \leq \frac{1}{2}\}$, the pairing of the constant term satisfies $\E_r^c \odot f(g_0) = O(1)$.
\end{lemma}
This lemma is used to handle reducible forms, for which $\theta_0 = 0$.
\begin{proof}
 As $t \to \infty$ in the cusp, the non-constant part of the Eisenstein series decays exponentially in $t$, which is faster than any polynomial.  In the equation $h g_0$ of the convolution, write
 $h=  k_\theta a_t n_u$, then one obtains 
 $$
  hg_0 =k_\theta a_t n_u \left(n_{-x_0} a_{\frac{1}{t_0}}k_{-\theta_0}\right).
 $$
By bi-$K$-invariance of $f$,
$f(hg_0) = f\left(n_{\frac{u-x_0}{t^2}}a_{\frac{t}{t_0}}\right)$.
Set $s=t/t_0$ and $v=(u-x_0)/t^2$. The resulting change of variables gives $du=t^2\,dv$. Thus the integral can be estimated using the decay of $f(n_v a_s)$. Using the Schwartz bound, for every $N\ge 0$, 
\begin{equation}
\int_\mathbb{R} |f(n_v a_s)|\,dv\ll_{f,N}(s+s^{-1})^{-N},\quad s>0.\label{eq:Schwartz}
\end{equation}

\ 

The non-constant term claim now follows due to the fact that $f$ is Schwarz class, e.g. split on whether $t \geq \sqrt{t_0}$ or not. 

One can bound with $\E_r^n$ in the case $t\ge\sqrt{t_0}$ is large: for suitable constants $C$, $c>0$, and some fixed exponent $M$, $$|\E_r^n((k_\theta a_t n_u)^t)|\le C(1+t)^M e^{-ct^2}, \quad t\ge 1.$$ Since $t_0\ge \frac{1}{2}$, the contribution is bounded by an integral of the form $$C\int_{\sqrt{t_0}}^\infty t^M e^{-ct^2}\left(\int_\mathbb{R} |f(n_va_{t/t_0})|\,dv\right) t^2\frac{dt}{t}.$$ Consequently, the contribution to integral when $t\ge\sqrt{t_0}$ is $O_{f,A}(t_0^{-A})$ for every $A>0$.

For $t<\sqrt{t_0}$, one can bound the integral with $f$. In this case, $s=\frac{t}{t_0}<t_0^{-1/2}$. For large $t_0$ the Schwartz bound gives $$\int_\mathbb{R} |f(n_v a_{t/t_0})|\,dv\ll_{f,N}\left(\frac{t}{t_0}\right)^N.$$ Also the non-constant Eisenstein series has a polynomial growth when $0<t\le 1$: for some $M\ge 0$, $$\sup_{\theta,u}|\E_r^n(k_\theta a_t n_u)|\ll_r t^{-M},$$ and exponential decay when $t\ge 1$: $$\sup_{\theta,u}|\E_r^n((k_\theta a_t n_u)^t) |\ll_r e^{-ct^2}$$ for some $c>0$. 

It suffices to consider $t_0\ge1$ since the remaining bounded range can be absorbed into the implied constant. Since $s=t/t_0<t_0^{-1/2}\le 1$, from \eqref{eq:Schwartz}, $$\int_\mathbb{R} |f(n_va_s)|\,dv\ll_{f,N}s^N, \quad 0<s\le 1.$$ Integrating over $t<\sqrt{t_0}$, for $0<t\le 1$, the integral is majorized by $\int_0^1 t^{-M}\left(\frac{t}{t_0}\right)^N\frac{dt}{t}=\frac{t_0^{-N}}{N-M}$ provided $N>M$, and for $1<t<\sqrt{t_0}$ is majorized by $t_0^{-N}\int_1^{\sqrt{t_0}} e^{-ct^2} t^N\frac{dt}{t}\ll t_0^{-N}$ where all implied constants depend on $f,N,r$.

\ 

To handle the case of the constant term, in which $t_0$ is bounded, again write $f(h g_0) = f\left(\frac{u-x_0}{t^2}a_{\frac{t}{t_0}}\right)$ and now use that the constant term $\E_r^c$ is bounded for $t = O(1)$ and grows polynomially, $O(t)$.  After integrating in the unipotent coordinate, the Schwartz bounds yield decay faster than any polynomial in both $t$ and $t^{-1}$, up to fixed powers arising from the change of variables. Thus the integral is majorized by $$ C_f\int_0^\infty (1+t)^M(1+t+t^{-1})^{-N}\frac{dt}{t},$$ with $N$ chosen sufficiently large. This integral is finite and independent of $t_0$. 
\end{proof}

\section{Cubic rings and binary cubic forms}

A cubic ring $R$ over $\zed$ is a free rank three $\zed$ module with a ring multiplication.  Delone-Fadeev and Gan-Gross-Savin established a discriminant-preserving bijection between cubic rings up to isomorphism and the space $\sym^3(\zed^2)$ of binary cubic forms. In the identification, maximal cubic rings whose associated form is irreducible over $\zed$ correspond with rings of integers in cubic number fields.

Given a form $x (v,w) = av^3 + bv^2w + cvw^2 + dw^3$ in the space $V_{\bR}$ of real binary cubic forms, $g \in  \GL_2(\bR)$ acts by $g\cdot x(v,w) = x((v,w)g)$. There is a bilinear pairing 
$$
 \langle x, y\rangle = x_4y_1 - \frac{1}{3}x_3y_2 + \frac{1}{3}x_2y_3 - x_1y_4.
$$
Let $P(x) = x_2^2x_3^2 +18x_1x_2x_3x_4 -4x_1x_3^3 -4x_2^3x_4 -27x_1^2x_4^2$ be the discriminant.  
The discriminant scales under the action by a factor of $\chi(g)=\det(g)^6$. The space $V_{\bR}$ decomposes into two open orbits $V_+$ and $V_-$ having positive and negative discriminant, and the singular set $S$ where the discriminant is 0.   
We identify $V_+ = G^+ \cdot x_+$, $V_- = G^+ \cdot x_-$ by choosing base points
\begin{equation}
 x_+ = \left(0,\frac{3}{(108)^{\frac{1}{4}}},0,\frac{-1}{(108)^{\frac{1}{4}}}\right), \qquad x_- = \left(0,\frac{1}{\sqrt{2}},0,\frac{1}{\sqrt{2}}\right).
\end{equation}
The point $x_+$ has stabilizer of order 3 generated by the rotation of $\frac{2\pi}{3}$, while $x_-$ has trivial stabilizer.  
We use several times that $d_\lambda n_u a_t k_\theta \cdot x_+$ has first coefficient $a = \frac{3 \cos^2 \theta \sin \theta - \sin^3 \theta}{108^{\frac{1}{4}}} \lambda^3 t^3$ proportional to $t^3$, and similarly  $d_\lambda n_u a_t k_\theta \cdot x_-$ has first coefficient $a$ equal to $\frac{\sin \theta}{\sqrt{2}}\lambda^3 t^3$ while $d_\lambda n_u a_t \cdot x_+$ has $a = 0$ and $b = \frac{3}{108^{\frac{1}{4}}} \lambda^3 t$ while $d_\lambda n_u a_t \cdot x_-$ has $a = 0$ and $b = \frac{1}{\sqrt{2}}\lambda^3 t$.  Since we choose $\lambda \asymp X^{\frac{1}{12}}$ for a discriminant of size $\asymp X$, if the resulting form is integral, it follows that if $a \neq 0$ then $|a| \geq 1$ implies $t \gg X^{-\frac{1}{12}}$ while if $a = 0$ but $b \neq 0$ then $|b| \geq 1$ implies $t \gg X^{-\frac{1}{4}}$.  Thus in either case there is a bound on the points that can occur in the Siegel set $\fS$ when representing integral forms of discriminant of size $X$.

We have the group integrals (\cite{S72}, Proposition 2.4)
$$
 \int_{g \in G^+} f(g \cdot x_+)dg = {4\pi} \int_{V_+} f(x) \frac{dx}{P(x)}
$$
and
$$
 \int_{g \in G^+} f(g \cdot x_-) dg = {12\pi} \int_{V_-} f(x) \frac{dx}{|P(x)|}.
$$
In particular,
\begin{align*}
 \int_{V_+} f(x) dx &= {4\pi} \int_{g \in G^+} f(g \cdot x_+) \chi(g) dg,\\
 \int_{V_-} f(x) dx &= {12\pi} \int_{g \in G^+} f(g \cdot x_-) \chi(g) dg.
\end{align*}
The Fourier transform is defined by
$$
 \hat{f}(\xi) = \int_{V_\bR} f(x) e^{-2\pi i \langle x, \xi\rangle}dx.
$$

In this paper it suffices to restrict attention to test functions which are bi-$K$-invariant and factor through the determinant.  Let $f_G$ be the function on $G^1$, $f_G(g) = \exp\left(-\tr g^t g\right)$, and extend $f_G$ to $G^+$ independent of the determinant.  Let $f_D \in C_c^\infty(\bR^+)$.  Define $f_-$ supported on $V_-$ and $f_+$ supported on $V_+$ by
$$
 f_-(g\cdot x_-) = f_G(g)f_D(\chi(g)), \qquad f_+(g \cdot x_+) = f_G(g)f_D(\chi(g)).
$$
For such functions, the Fourier transform $\hat{f}$ is left-$K$-invariant, since
\begin{align*}
 \hat{f}(k_\theta \cdot \xi) &= \int_{V_{\bR}} f(x) e^{-2\pi i \langle x, k_{\theta}\cdot \xi \rangle} dx\\
 &= \int_{V_{\bR}} f(x) e^{-2\pi i \langle k_{-\theta} x, \xi\rangle} dx = \hat{f}(\xi).
\end{align*}

The following integrals of the Fourier transform of $f_{\pm}$ are used.

Define $\Sigma_2(f, z) = \int_{0}^\infty f(0,0,0,t)t^{z-1} dt.$  This satisfies \begin{equation}\Sigma_2(f^t, z) = t^{-z} \Sigma_2(f,z).\end{equation}
\begin{lemma}\label{Sigma_2_lemma} We have 
 \begin{align*} 
  \Sigma_2\left(\hat{f}_-,z\right) &= 2^{-\frac{z}{2}} \pi^{1-z} \cos\left(\frac{\pi z}{2} \right)\frac{\Gamma(z)\Gamma\left(\frac{1-z}{2} \right)}{\Gamma\left(1 - \frac{z}{2}\right)}\tilde{f}_D\left(1 - \frac{z}{4}\right) K_{\frac{1-3z}{2}}(2),\\
  \Sigma_2\left(\hat{f}_+,z\right)& = 3^{\frac{3z}{4}-1}2^{-\frac{z}{2}} \pi^{1-z} \cos\left(\frac{\pi z}{2} \right)\frac{\Gamma(z)\Gamma\left(\frac{1-z}{2} \right)}{\Gamma\left(1 - \frac{z}{2}\right)}\tilde{f}_D\left(1 - \frac{z}{4}\right) K_{\frac{1-3z}{2}}(2).
 \end{align*}
\end{lemma}
\begin{proof}
 Note that the left-$K$-invariance causes $f_{\pm}$ to be even. Calculate
 \begin{align*}
  \Sigma_2\left(\hat{f}_{\pm}, z\right) &= \int_0^\infty \hat{f}_{\pm}(0,0,0,t)t^{z}\frac{dt}{t}\\
  &= \int_0^\infty t^z \frac{dt}{t} \int_{x \in V_{\bR}} f_{\pm}(x) \cos(2\pi  x_1 t)dx\\
  &= (2\pi)^{-z}\cos\left(\frac{\pi z}{2}\right) \Gamma(z)\int_{x \in V_{\bR}} f_{\pm}(x) |x_1|^{-z}dx.
  \end{align*}
  In the case of $f_-$, write this as 
  \begin{align*}
  &(2\pi)^{-z}\cos\left(\frac{\pi z}{2}\right) \Gamma(z) \Bigl(12\pi \int_{g \in G^+}f_G(g) |(g\cdot x_-)_1|^{-z} f_D(\chi(g))\chi(g) dg \Bigr) .
 \end{align*}
In the $ANK$ decomposition, Haar measure is $\frac{dt}{t}du d\theta$.  In the negative discriminant case, $x_- = \frac{1}{\sqrt{2}}(0, 1, 0, 1)$ and the first coefficient of $a_t n_u k_\theta \cdot x_-$ is $\frac{t^3 \sin 2\pi \theta}{\sqrt{2}}$.  
Using that $f$ is right $K$ invariant, integrate in $\theta$ using 
$$
 \int_0^1 |\sin(2\pi \theta)|^{-z} d\theta = \frac{\Gamma\left( \frac{1-z}{2} \right)}{\sqrt{\pi}\Gamma\left(1 - \frac{z}{2} \right)}.
$$
Thus the negative discriminant case is given by
\begin{align*}
 &2^{\frac{z}{2}}(2\pi)^{-z} \cos\left(\frac{\pi z}{2} \right) \Gamma(z) 12 \pi \\&\times \int_0^\infty \frac{d\lambda}{\lambda} f_D(\lambda^{12}) \lambda^{12-3z} \int_0^1 |\sin(2\pi \theta)|^{-z} d\theta \int_{-\infty}^\infty du \int_0^\infty \frac{dt}{t} t^{-3z} \exp\left(-t^2 -\frac{1}{t^2} -\frac{u^2}{t^2} \right)\\
 &= 2^{-\frac{z}{2}} \pi^{1-z} \cos\left(\frac{\pi z}{2} \right)\frac{\Gamma(z)\Gamma\left(\frac{1-z}{2} \right)}{\Gamma\left(1 - \frac{z}{2}\right)}\tilde{f}_D\left(1 - \frac{z}{4}\right) K_{\frac{1-3z}{2}}(2).
\end{align*}

In the case of $f_+$, write the integral as
  \begin{align*}
  &(2\pi)^{-z}\cos\left(\frac{\pi z}{2}\right) \Gamma(z) \Bigl(4\pi \int_{g \in G^+}f(g) |(g\cdot x_+)_1|^{-z} \chi(g) dg \Bigr) .
 \end{align*}
In the positive discriminant case, use that the stabilizer of $x_+$ is the rotation group generated by rotation by $\frac{2\pi}{3}$.  The first coefficient of $a_t n_u k_\theta \cdot x_+$ is $\frac{t^3 \sin 6\pi \theta }{(108)^{\frac{1}{4}}}$.  Thus the positive discriminant case is given by $3^{\frac{3z}{4}-1}$ times the integral in the negative discriminant case.  This obtains the lemma.
\end{proof}

Define 
\begin{equation}
 \Sigma_3(f, s) = \int_0^\infty dt\int_{-\infty}^\infty du f(0,0, t,u)t^{s-1}.
\end{equation}
This satisfies 
\begin{equation}
\Sigma_3\left(f^t, s\right) = t^{-s-1}\Sigma_3(f,s).
\end{equation}

\begin{lemma}\label{Sigma_3_lemma}
 Let, for $x \in \bR^3$, $f_{\pm,0}(x) = f_{\pm}(0,x)$. We have
 $$
  \Sigma_3\left(\hat{f}_{\pm}, s\right) = 3^s\pi^{-s+\frac{1}{2}} \frac{\Gamma\left(\frac{s}{2}\right)}{2\Gamma\left(\frac{1-s}{2}\right)} \int_{x = (x_2, x_3, x_4)} f_{\pm,0}(x) |x_2|^{-s}dx.
 $$

\end{lemma}

\begin{proof}
 Using the bilinear pairing $\langle x, y\rangle= x_4y_1 - \frac{1}{3}x_3y_2 +\frac{1}{3}x_2y_3 - x_1y_4$, and the fact that $f$ is even, calculate
 \begin{align*}
  \Sigma_3\left(\hat{f}_{\pm},s\right) &= \int_{x = (x_2, x_3, x_4)} \int_0^\infty \frac{dt}{t} f_{\pm}(0, x_2, x_3, x_4)t^s \cos\left(2\pi  \frac{tx_2}{3}\right)dx\\
  &= 3^{s}(2\pi)^{-s}\cos\left(\frac{\pi s}{2}\right)\Gamma(s) \int_{x = (x_2, x_3, x_4)} f_{\pm,0}(x) |x_2|^{-s} dx.
 \end{align*}
Combined with the formula $\cos \left(\frac{\pi s}{2}\right) = \frac{\pi}{\Gamma\left(\frac{s+1}{2}\right)\Gamma\left(\frac{1-s}{2}\right)}$ and the formula 
$$
 \Gamma(s) = \frac{\Gamma\left(\frac{s}{2}\right)\Gamma\left(\frac{s+1}{2}\right)}{2^{1-s}\sqrt{\pi}}
$$
this proves the lemma.
\end{proof}

\begin{lemma} \label{Phi_0_lemma}
 We have
 \begin{align*}
  \int_{x= (x_2,x_3, x_4)}f_{-,0}(x)|x_2|^s dx &= 2^{\frac{1-s}{2}} \tilde{f}_D\left(\frac{3 + s}{4} \right)\sqrt{\pi} K_{\frac{s-2}{2}}(2),\\
  \int_{x = (x_2, x_3, x_4)} f_{+,0}(x)|x_2|^s dx &=3^{\frac{s-1}{4}}2^{\frac{1-s}{2}} \tilde{f}_D\left(\frac{3 + s}{4} \right)\sqrt{\pi} K_{\frac{s-2}{2}}(2).
 \end{align*}
\end{lemma}
\begin{proof}
 Let $g_2 \cdot x$ denote the action of $G^+$ on binary quadratic forms.  Identify $x_+, x_-$ with points in the space of binary quadratic forms by dropping the first coefficient.  The action of $(d_\lambda a_t n_u)_2$ on $x_{\pm}$ is given by
 \begin{align*}
  (d_\lambda a_t n_u)_2 \cdot x_- &= \left( \frac{\lambda^2 t^2}{\sqrt{2}}, \frac{2\lambda^2 u}{\sqrt{2}}, \frac{\lambda^2 (1 + u^2)}{\sqrt{2}t^2}\right),\\
  (d_\lambda a_t n_u)_2 \cdot x_+ &= \left(\frac{3\lambda^2t^2}{(108)^{\frac{1}{4}}}, \frac{6 \lambda^2 u}{(108)^{\frac{1}{4}}}, \frac{\lambda^2(-1 + 3u^2)}{(108)^{\frac{1}{4}}t^2} \right).
 \end{align*}
 In the case of $V_-$, the volume form is equal to 
 $$|dx_1 \wedge dx_2 \wedge dx_3| = \frac{2^{\frac{5}{2}}\lambda^5 }{t} |dt \wedge du \wedge d\lambda|,$$
 while on $V_+$ the volume form is given by
 $$
  |dx_1 \wedge dx_2 \wedge dx_3| = \frac{2^{\frac{5}{2}}\lambda^5}{3^{\frac{1}{4}}t} |dt \wedge du \wedge d\lambda|.
 $$
Thus the integral over $V_-$ is given by
\begin{align*}
& 2^{\frac{7-s}{2}} \int_{0}^\infty \frac{d\lambda}{\lambda} \lambda^{6+2s} \int_0^\infty \frac{dt}{t} t^{2s} \int_{-\infty}^\infty du  f_{-,0}((d_\lambda a_t n_u)_2 \cdot x_-)\\
&= 3\cdot 2^{\frac{5-s}{2}} \int_0^\infty \frac{d\lambda}{\lambda} \lambda^{9 + 3s} \int_0^\infty \frac{dt}{t} t^{-3+s} \int_{-\infty}^\infty du f_{-,0} \left( \left( d_{\sqrt{\frac{\lambda^3}{t}}} a_t n_u\right)_2 \cdot x_-\right)\\
&=3\cdot 2^{\frac{5-s}{2}} \int_0^\infty \frac{d\lambda}{\lambda} \lambda^{9 + 3s} \int_0^\infty \frac{dt}{t} t^{-3+s} \int_{-\infty}^\infty du f_- \left( \left( d_{\lambda} a_t n_u\right)_3 \cdot x_-\right)\\
&= 3 \cdot 2^{\frac{5-s}{2}} \int_0^\infty \frac{d\lambda}{\lambda} \lambda^{9 + 3s} f_D(\lambda^{12}) \int_0^\infty \frac{dt}{t} t^{-3+s} \int_{-\infty}^\infty du \exp\left(-t^2 - \frac{1}{t^2} - \frac{u^2}{t^2}\right)\\
&= 2^{\frac{1-s}{2}} \tilde{f}_D\left(\frac{3 + s}{4} \right)\sqrt{\pi} K_{\frac{s-2}{2}}(2).
\end{align*}
The integral over $V_+$ differs from the above by a factor of $3^{\frac{s-1}{4}}$.
\end{proof}

Write $L$ for space of integral binary cubic forms and $\hat{L}$ for the dual forms, which have middle coefficients divisible by 3. For each $m  \neq 0$ let $h(m)$ be the class number of forms of discriminant $m$ and $\hat{h}(m)$ the class number of dual forms.  
Shintani obtained the following lemma regarding singular integral forms.
\begin{lemma}\label{fibration_lemma} The singular forms  $\hat{L}_0$ are the disjoint union
\begin{align}
\hat{L}_0 &= \{0\}\sqcup L_0(I)\sqcup \hat{L}_0(II)
\end{align}
where
\begin{align}
L_0(I) &= \bigsqcup_{m=1}^\infty \bigsqcup_{\gamma\in\Gamma/\Gamma\cap N}\{\gamma\cdot (0,0,0,m)\},\\ 
\notag \hat{L}_0(II) &= \bigsqcup_{m=1}^\infty\bigsqcup_{n=0}^{3m-1}\bigsqcup_{\gamma \in \Gamma}\{\gamma \cdot (0,0,3m,n)\}.
\end{align}
\end{lemma}

The forms of positive discriminant break into two classes, the first of which
have stability group in $\Gamma$ which is trivial, and the second having
stability group of order 3.

For $m \neq 0$ let $\{g_{i,m}\}_{1 \leq i \leq h(m)} \subset G^+,$ (resp.\ $\{\hat{g}_{i,m}\}_{1 \leq i \leq \hat{h}(m)}$) be such that 
 \begin{equation}\{x_{i,m} = g_{i,m}\cdot x_{\sgn \;m}\}_{1 \leq i \leq h(m)}\end{equation} (resp.\ $\{\hat{x}_{i,m} = \hat{g}_{i,m}\cdot x_{\sgn \;m}\}_{1 \leq i \leq \hat{h}(m)}$)  are
representatives for the classes of binary cubic forms of discriminant $m$ (resp.\ classes of dual forms). The points $g_{i,m}$ naturally identify the cubic ring associated with $x_{i,m}$ with its lattice shape  in $\Gamma \backslash G^1/K$.    Set $\Gamma(i,m)< \Gamma$ the
stability group of $x_{i,m}$, similarly $\hat{\Gamma}(i,m)$. We observe the convention that the matrices $g_{i,m}$ lie in a fundamental domain for $\Gamma \backslash G^1$ contained in the Siegel set $\fS = \{n_u a_t k_\theta: \theta \in \bR, t < 2, |u| \leq \frac{1}{2}\}$.  Thus $(g_{i,m}^{-1})^t$ lies in the standard fundamental domain for e.g. the upper half plane modulo $\Gamma$, contained in the Siegel set $\fS' = \{\nu_x a_t k_\theta: \theta \in \bR, t> \frac{1}{2}, |x| \leq \frac{1}{2}\}$.  In practice the automorphic forms are evaluated on $(g_{i,m}^{-1})^t$ and are applied on the standard fundamental domain.

\subsection{Reducible forms}

Due to the growth of the Eisenstein series in the cusp, and the fact that a significant contribution of forms reducible over $\zed$ occur in the cusp, in this work it is necessary to give special treatment to the reducible forms.  Evidently every non-degenerate reducible form $x$ is equivalent under $\Gamma$ to a form of type $\sR = \{bv^2w + cvw^2 + dw^3: 0 \leq c < 2b\}$.  If $c^2 - 4bd$ is not a square, then $x$ has a unique representative in $\sR$, while if $c^2 - 4bd$ is a square then there are one or three forms equivalent to $x$ in $\sR$ according as the stabilizer subgroup of $x$ in $\Gamma$ has size three or one, see the discussion in \cite{S75}, pp.\ 45-46.

Following Shintani, \cite{S75}, our treatment of reducible forms identifies $\sR$ with the prehomogeneous vector space $\sym^2(\bR^2)$ of binary quadratic forms acted on by the group of lower triangular matrices $$B^+ = \left\{ \begin{pmatrix} g_{11}& \\g_{21} & g_{22} \end{pmatrix}: g_{11}, g_{22} > 0\right\}.$$
Given a binary quadratic form $x = av^2 + bvw + cw^2$ associated to symmetric matrix $Q = \begin{pmatrix} a & \frac{b}{2}\\ \frac{b}{2} & c\end{pmatrix}$ the action of $g \in B^+$ is written $\rho(g) \cdot x$ or $g \cdot x$ for short, and maps
$$
 Q \mapsto gQg^t.
$$
There are two invariants, the discriminant $D = b^2 - 4ac$ and the first coefficient $x_1=a$.  Let $\chi(g) = (\det g)^2$ and $\chi_1(g) = g_{11}^2$.  We have$$\Disc( g \cdot x) = \chi(g) \Disc(x), \qquad (g \cdot x)_1 = \chi_1(g) x_1.$$ The contragredient representation of $\rho$ is $\rho^*(g) = \frac{1}{(\det g)^2} \rho(g)$.
When $g \in B^+$ is represented in coordinates $g = \begin{pmatrix} t_1 & \\ u & t_2 \end{pmatrix}$, Shintani uses the Haar measure  $dg = t_1^{-2} t_2^{-1} dt_1dt_2 du$ which satisfies 
\begin{equation}
\int_{B^+} f\left(\begin{pmatrix} t_1 & \\ u & t_2 \end{pmatrix} \right) \frac{dt_1}{t_1^2}\frac{dt_2}{t_2}du = 2 \int_0^\infty \frac{d\lambda}{\lambda}\int_0^\infty \frac{dt}{t^3} \int_{-\infty}^\infty du f(d_\lambda a_t n_u).
\end{equation}
We use the latter normalization, so that the orbital zeta functions in our work differ by a factor of $\frac{1}{2}$ from Shintani's.
There is a bilinear pairing on $\sym^2(\bR^2)$ given by $[x,y] = x_1y_3 - \frac{1}{2} x_2y_2 + x_3y_1$. Let $L_r = \sym^2(\zed^2)$ with dual forms $\hat{L}_r = \sym^2(\zed^2)^*$.  Thus dual integral forms have even middle coefficient.

The singular points of the representation $(V, B^+, \rho)$ are the union $S \cup S_1$ where $S = \{\Disc = 0\}$, $S_1 = \{x_1 = 0\}$. Let $L_{0,r}$ be the singular integral forms, and $\hat{L}_{0,r}$ the singular dual integral forms.
\begin{lemma}\label{reducible_singular_fibration_lemma}
 The set $\hat{L}_{0,r}$ is the disjoint union
 \begin{equation}
  \hat{L}_{0,r}=\{(0,0,0)\} \sqcup L_{0,r}(I)\sqcup \hat{L}_{0,r}(II)\sqcup \hat{L}_{0,r}(III)\end{equation}
  where
  \begin{align*}
 L_{0,r}(I)&= \{(0,0,\ell): \ell \in \zed \setminus \{0\}\},\\
 \hat{L}_{0,r}(II) &= \bigsqcup_{m \neq 0} \bigsqcup_{0 \leq n < |2m|} \bigsqcup_{\gamma \in B_\zed^+}  \{\gamma \cdot (0, 2m, n)\},\\
 \hat{L}_{0,r}(III)&= \bigsqcup_{\ell \neq 0} \bigsqcup_{b \neq 0} \bigsqcup_{\substack{0 \leq d < b, \\(b,d)=1}} \bigsqcup_{\gamma \in B_{\zed}^+} \{\gamma \cdot\ell(b^2, 2bd, d^2)\}.
  \end{align*}

\end{lemma}

\begin{proof}
 See \cite{S75}, Lemma 4.
\end{proof}

Let $V_{2,\pm}= \{x \in \sym^2(\bR): \pm \Disc(x) > 0\}$. Define 
$$
 \Phi_\pm(f,s_1,s_2) = \int_{V_{2, \pm}} f(x) |x_1|^{s_1} |\Disc(x)|^{s_2} dx.
$$
In the case that $f = f_{\pm, 0}$ from the previous section, and $s_2 = 0$, by Lemma \ref{Phi_0_lemma},
\begin{align*}
 3^s \pi^{-s + \frac{1}{2}} \frac{\Gamma\left(\frac{s}{2}\right)}{2\Gamma\left(\frac{1-s}{2}\right)}\Phi_+(f_{+,0}, -s, 0) &= \Sigma_3\left(\hat{f}_+, s\right),\\
 3^s \pi^{-s + \frac{1}{2}} \frac{\Gamma\left(\frac{s}{2}\right)}{2\Gamma\left(\frac{1-s}{2}\right)}\Phi_-(f_{-,0}, -s, 0) &= \Sigma_3\left(\hat{f}_-, s\right).
\end{align*}

Also,
$$
 \Sigma(f,s) = \int_{-\infty}^\infty \int_{-\infty}^\infty f(t, 2u, t^{-1}u^2) |t|^{s-1}dt du.
$$
This is holomorphic in $\RE(s)>-\frac{1}{2}$.

\section{Conditions of maximality}

A cubic ring is maximal if it is notproperly contained in another cubic ring.  The condition of maximality may be checked locally, that is a cubic ring $R$ is maximal if and only if $R \otimes \zed_p$ is maximal for every prime $p$.  At the level of binary cubic forms, Delone and Fadeev identify the non-maximal set as a list of congruences modulo $p^2$.  A binary cubic form $x = av^3 + bv^2w + cvw^2 + dw^3$ is non-maximal at a prime $p$ if one of the following two conditions is met:
\begin{enumerate}
 \item The coefficients of $x$ are divisible by $p$.
 \item The form is $\GL_2(\zed)$ equivalent to a form with $p^2|a$ and $p|b$.
\end{enumerate}

Let $\sN_p$ be the indicator function that a form $x \in V_{\zed/p^2\zed}$ is non-maximal at prime $p$. Define the discrete Fourier transform of $\sN_p$ by, for $\xi \in (\zed/p^2 \zed)^4$,
\begin{equation}
 \hat{\sN}_p(\xi) = \frac{1}{p^8} \sum_{x \in V_{\zed/p^2 \zed}} \sN_p(x) e\left(\frac{\langle x, \xi \rangle}{p^2} \right).
\end{equation}

\begin{lemma}
 Let $p$ be prime. For $m, n \in \zed$ we have the evaluations
 \begin{align*}
  \hat{\sN}_p(0,0,0,m) &= \left\{\begin{array}{lll} p^{-2}+p^{-3}-p^{-5}, && p^2|m\\
   p^{-3} -p^{-5}, && p^2 \nmid m                               
\end{array}\right.\\
\hat{\sN}_p(0,0,3m,n) &= \left\{\begin{array}{lll} p^{-2} +p^{-3}-p^{-5}, && p^2|(m,n)\\
p^{-3}-p^{-5}, && p^2\nmid (m,n), p|m\\
0, &&p \nmid m
 \end{array} \right..
 \end{align*}

\end{lemma}
\begin{proof}
 For $p = 2, 3$ this was checked in SciPy.  For $p > 3$ this is proved in \cite{TT13b}.
\end{proof}

Extend $\sN_p$ to $\sN_q$ multiplicatively when $q = p_1p_2...p_n$ is square-free, that is
\begin{equation}
 \sN_q(x) = \prod_{i=1}^n \sN_{p_i}(x).
\end{equation}
We have, \cite{H19},
$
 \hat{\sN}_q(\xi) = \prod_{p|q} \hat{\sN}_p(\xi).
$
The following theorem is proved in \cite{TT13b}. 
\begin{theorem}
 Uniformly in $q$ and $Y$, for all $\epsilon>0$,
 $$
  \sum_{0 < |m| \leq Y}\sum_{i=1}^{\hat{h}(m)} \left|\hat{\sN}_q(\hat{x}_{i,m})\right| \ll_\epsilon q^{-7+\epsilon}Y.
 $$

\end{theorem}

The reducible forms which are non-maximal are classified in the following lemma.

\begin{lemma}
 Those forms $bv^2w + cvw^2 + dw^3$, $b,c,d \in \zed^3$ which are non-maximal at $p > 2$ have one of the following forms.
 \begin{enumerate}
  \item $b \equiv 0 \bmod p$.
  \item $\{\beta w(v + \delta w)^2: \beta \in (\zed/p^2\zed)^\times,  \delta \in \zed/p^2 \zed\}$.
 \end{enumerate}
If $p=2$ then replace the second item with
\begin{enumerate}
 \item [(2')] $\{\beta w (v + \delta w)(v +( \delta + 2\gamma)w): \beta \in \{1,3\}, \delta \in \{0,1\}, \gamma \in \{0,1\}\}$.
\end{enumerate}

\end{lemma}

\begin{proof}
 The criterion for a binary cubic form $aw^3 + bw^2v + cwv^2 + dv^3$ to represent a cubic ring non-maximal at $p$ is that either the form is divisible by $p$, or the form is $\GL$ equivalent to a form with $p^2|a$ and $p|b$.  When $a = 0$, the first condition is contained in the second, so we consider only this case.
 
 Consider those forms equivalent to a form $x = p^2 a v^3 + pb v^2w + c vw^2 + dw^3$.  If $p$ divides $c$ and $d$ then any equivalent form has coefficients divisible by $p$.  This is absorbed in the first case of the lemma, which is evidently non-maximal.  If, instead, $c \equiv 0 \bmod p$ but $d \not \equiv 0 \bmod p$, make the change of coordinates 
 $$
  v \mapsto \alpha v + \beta w, \qquad w \mapsto \gamma v + \delta w.
 $$
This maps $x$ to $g \cdot x \equiv d(\gamma v + \delta w)^3 \bmod p$.  For the first coefficient to vanish, $\gamma \equiv 0 \bmod p$.  This obtains a form of type (1) from the lemma.

Finally, suppose $c \not \equiv 0 \bmod p$.  After making a change of type 
$$
 v \mapsto \alpha v + \beta w, \qquad w \mapsto \delta w
$$
the form $x$ is equivalent to a form
$$
 x' \equiv bp v^2 w + vw^2 \bmod p^2.
$$
Under the change 
$$
 v \mapsto \alpha v + \beta w, \qquad w \mapsto \gamma v + \delta w
$$
this maps to a form equivalent modulo $p^2$ to
$$
 bp(\alpha v + \beta w)^2 (\gamma v + \delta w)+ (\alpha v + \beta w) (\gamma v + \delta w)^2 \bmod p^2. 
$$
The first coefficient is equivalent to $bp\alpha^2 \gamma + \alpha \gamma^2 \bmod p^2$.
If $p|\gamma$ then the transformed form is of type (1). If $p \nmid \gamma$ then $p^2|\alpha$.  The transformed form is now equivalent to 
$$
 \beta w (\gamma v + \delta w)^2 + bp \beta^2 w^2 (\gamma v + \delta w) \bmod p^2.
$$
By adjusting the form by a form divisible by $p^3$ all such forms are attained.  These forms are non-maximal and are of type (2) if $p > 2$ and (2') in the case $p = 2$.
\end{proof}

Write the non-maximal set as the disjoint union $\sM_1 \sqcup \sM_2$ where
\begin{align*}
 \sM_1 &= \{b v^2w + cvw^2 + dw^3: p|b\}\\
 \sM_2 &= \{\beta w (v + \delta w)^2: \beta \in (\zed/p^2\zed)^\times, \delta \in \zed/p^2\zed\},
\end{align*}
and$$
 \sM_2' = \{\beta w (v + \delta w)(v +( \delta + 2\gamma)w): \beta \in \{1,3\}, \delta \in \{0,1\}, \gamma \in \{0,1\}\}
$$
in the case $p = 2$.

Thus $\sN_p^r = \one_{\sM_1} + \one_{\sM_2}$, identified as a function on $\sym^2(\zed)$. For square-free $q$, let $\sN_q^r(x) = \prod_{p|q} \sN_p^r(x)$.  Since $\sM_1$ and $\sM_2$ are invariant under dilation, 
$$
 \hat{\sN}_q^r(\xi) = \prod_{p|q} \hat{\sN}_p^r(\xi).
$$

Let the discrete Fourier transform of a function on $(\zed/p^2\zed)^3$ as 
\begin{equation}
 \hat{f}(\xi) = \frac{1}{p^6}\sum_{x \in (\zed/p^2\zed)^3} f(x) e\left(\frac{x_1\xi_3 -\frac{1}{2}x_2\xi_2 + x_3\xi_1}{p^2} \right).
\end{equation}
The discrete Fourier transforms of the indicator functions of $\sM_1$ and $\sM_2$, $\sM_2'$ satisfy the following.

\begin{lemma}
 The discrete Fourier transforms of $\sM_1$ is given by
 \begin{align*}
  \hat{\one}_{\sM_1}(\xi_1, 2\xi_2, \xi_3) = \left\{\begin{array}{lll} \frac{1}{p} && \xi_1 = \xi_2 = 0, p|\xi_3\\ 0 && \text{otherwise}\end{array}\right.
 \end{align*}
 and, for $p = 2$,
 \begin{align*} 
  \hat{\one}_{\sM_2'}(\xi_1, 2\xi_2, \xi_3) &= \frac{1}{16} \Biggl[\one(\xi_3 \equiv 0 \bmod 4 \wedge \xi_2 \equiv 0 \bmod 2) - \one(\xi_3 \equiv 2 \bmod 4 \wedge \xi_2 \equiv 0 \bmod 2)\\& + \one(\xi_3 - 2\xi_2 + \xi_1 \equiv 0 \bmod 4 \wedge \xi_1 + \xi_2 \equiv 0 \bmod 2) \\&- \one(\xi_3-2\xi_2 + \xi_1 \equiv 2 \bmod 4 \wedge \xi_1 + \xi_2 \equiv 0 \bmod 2) \Biggr]. 
 \end{align*}
If $\xi = p \xi'$ is divisible by $p$, the sum is
\begin{align*}
 \hat{\one}_{\sM_2}(p\xi_1', 2p\xi_2', p\xi_3') = \left\{\begin{array}{lll} \frac{p-1}{p^3} && \xi' \equiv 0 \bmod p \\ \frac{-1}{p^3} && \xi_1', \xi_2' \equiv 0 \bmod p, \xi_3' \not \equiv 0 \bmod p\\
\frac{\#\{\xi_1' \delta^2 - 2\xi_2' \delta +\xi_3' \equiv 0 \bmod p\} - 1}{p^3}&& \xi_1' \not \equiv 0 \bmod p\\
0 && \text{otherwise}\end{array}
 \right..
\end{align*}
If $\xi_1$ is divisible by $p$, but either $\xi_2$ or $\xi_3$ is not, the sum vanishes.  
If $\xi_1$ is not divisible by $p$, then the sum has value
\begin{align*}
 \hat{\one}_{\sM_2}(\xi) &=  \left\{\begin{array}{lll}\frac{p-1}{p^4} && \xi_1\xi_3 \equiv \xi_2^2 \bmod p^2\\ \frac{-1}{p^4} && \xi_1\xi_3 \not \equiv \xi_2^2 \bmod p^2, \xi_1\xi_3 \equiv \xi_2^2 \bmod p\\ 0 && \xi_1 \xi_3 \not \equiv \xi_2^2 \bmod p \end{array}  \right..
\end{align*}
\end{lemma}

\begin{proof}
 The evaluation of $\hat{\one}_{\sM_1}$ is straightforward.  In the case that $p = 2$, the sum over $\sM_2'$ is
 \begin{align*}
  &\frac{1}{64} \sum_{\beta \in \left(\zed/4\zed\right)^\times} \sum_{\delta \in \{0,1\}}\sum_{\gamma \in \{0,1\}} e\left(\frac{\beta(\xi_3 - 2\xi_2(\delta+\gamma) + \xi_1 (\delta^2+2\gamma \delta))}{4} \right)\\
  &= \frac{1}{32} \sum_{\beta \in (\zed/4\zed)^\times} e\left(\frac{\beta \xi_3}{4} \right)\one(\xi_2 \equiv 0 \bmod 2)\\& + \frac{1}{32} \sum_{\beta \in (\zed/4\zed)^\times} e\left(\frac{\beta(\xi_3 - 2\xi_2 + \xi_1)}{4} \right)\one(\xi_1 + \xi_2 \equiv 0 \bmod 2). 
 \end{align*}
This obtains the evaluation.

In the case $p \geq 3$, when $\xi = p \xi'$
\begin{align*}
 \hat{\one}_{\sM_2}(\xi') &= p^{-4}\sum_{\beta \in (\zed/p\zed)^\times} \sum_{\delta \in \zed/p\zed} e\left(\frac{\beta(\xi_3' - 2\delta \xi_2' + \delta^2 \xi_1')}{p} \right).
\end{align*}
If $\xi_1' \equiv 0 \bmod p$ then the sum vanishes unless $\xi_2' \equiv 0 \bmod p$.  In this case, the sum is $\frac{-1}{p^3}$ if $\xi_3' \not \equiv 0 \bmod p$ and $\frac{p-1}{p^3}$ if $\xi_3' \equiv 0 \bmod p$.  If $\xi_1' \not \equiv 0 \bmod p$, then the sum is
$\frac{\# \{\delta: \xi_1' \delta^2 - 2\xi_2' \delta + \xi_3' \equiv 0 \bmod p\} - 1}{p^3}$.

If $\xi_1$ is divisible by $p$, then the sum vanishes on summing over shifts of $\delta$ by $p$, unless $\xi_2$ is divisible by $p$.  If both $\xi_1$ and $\xi_2$ are divisible by $p$,  the sum now vanishes on summing in translates of $\beta$ by $p$ unless $\xi_3$ is divisible by $p$.  Thus assume $\xi_1$ is not divisible by $p$.  

Write $\delta = a + bp$ with $0 \leq a, b < p$ and sum in $b$.  This obtains
\begin{align*}
 \hat{\one}_{\sM_2}(\xi) &= p^{-6} \sum_{\beta \in (\zed/p^2\zed)^\times} \sum_{0 \leq a, b < p} e\left(\frac{\beta(\xi_3 -2(a + bp)\xi_2 + (a^2 + 2abp)\xi_1)}{p^2} \right)\\
 &= p^{-5} \sum_{\beta \in (\zed/p^2\zed)^\times} e\left(\frac{\beta( \xi_3 - \xi_2^2\xi_1^{-1})}{p^2} \right)\\
 &= \left\{\begin{array}{lll}\frac{p-1}{p^4} && \xi_1\xi_3 \equiv \xi_2^2 \bmod p^2\\ \frac{-1}{p^4} && \xi_1\xi_3 \not \equiv \xi_2^2 \bmod p^2, \xi_1\xi_3 \equiv \xi_2^2 \bmod p\\ 0 && \xi_1 \xi_3 \not \equiv \xi_2^2 \bmod p \end{array}  \right..
\end{align*}

\end{proof}

\begin{lemma} Summed over a complete period mod $q^2$, 
$$
 \sum_{\xi = (\xi_1, 2\xi_2, \xi_3) \bmod q^2} |\hat{\sN}_q^r(\xi)| = O(q^{1+\epsilon}).
$$

\end{lemma}

\begin{proof}
 This is the result of summing over all congruence classes mod $p^2$ in the lemma above for a given $p$, then multiplying over all $p |q$.  The corresponding statement holds trivially at 2.
\end{proof}

The previous lemma implies the following evaluation at singular forms.
\begin{lemma}
 For $p = 2$, when $(b,d) = 1$,
 \begin{align*}
  \hat{\sN}_{2}^r(\ell(b^2, 2bd, d^2)) &= \left\{\begin{array}{lll} \frac{5}{8} && 4|\ell\\
  \frac{3}{8} && 2\|\ell, 2|b\\
  \frac{1}{16} && 2\nmid \ell, b\\
   0 && \text{otherwise}\end{array}\right.\\
  \hat{\sN}_2^r(0,2m, n) &= \left\{\begin{array}{lll}\frac{5}{8} && 4|m,n\\ \frac{3}{8} && 4|m, 2\|n\\ \frac{1}{8} && 2\|m, 4|n\\ -\frac{1}{8} && 2\|m,n\\
  0 && \text{otherwise}\end{array}\right..
 \end{align*}
For $p > 2$, when $(b,d)=1$,
\begin{align*}
 \hat{\sN}_p^r(\ell(b^2, 2bd, d^2)) &= \left\{\begin{array}{lll}\frac{1}{p}+\frac{1}{p^2}-\frac{1}{p^3} && p^2|\ell\\ \frac{1}{p}-\frac{1}{p^3} && p\| \ell, p|b\\
\frac{1}{p^3}-\frac{1}{p^4} && p\nmid \ell, b\\
0 && \text{otherwise}                                             
\end{array}\right.\\
 \hat{\sN}_p^r(0,2m,n) &=\left\{\begin{array}{lll} \frac{1}{p}+\frac{1}{p^2} - \frac{1}{p^3} && p^2|m,n\\
 \frac{1}{p}-\frac{1}{p^3} && p^2|m, p\|n\\
 0 && \text{otherwise}
 \end{array}\right..
\end{align*}

\end{lemma}
Note that $\hat{\sN}_2^r(\ell(b^2, 2bd, d^2))$ is not $B_\zed^+$ invariant but that the other functions are invariant under $B_\zed^+$.

Introduce the orbital zeta functions in the space of binary quadratic forms,
\begin{equation}
 Z_{q,r}(f, s_1, s_2) = \int_{B^+/B_\zed^+} \chi_1(g)^{s_1}\chi(g)^{s_2} \sum_{x \in L_r} \sN_{q,r}(x) f(\rho(g)\cdot x) dg.
\end{equation}
Also, define
\begin{align*}
Z_{q,r}^+(f,s_1, s_2) &= \int_{B^+/B_\zed^+, \chi(g)>1} \chi_1(g)^{s_1}\chi(g)^{s_2} \sum_{x \in L_r'} \sN_{q,r}(x) f(\rho(g)\cdot x) dg,\\
\hat{Z}_{q,r}^+(\hat{f}, s_1, s_2) &= \int_{B^+/B_\zed^+, \chi(g)>1} \chi_1(g)^{s_1}\chi(g)^{s_2} \sum_{x \in \hat{L}_r'} \hat{\sN}_{q,r}(x) \hat{f}\left(\rho(g) \cdot \frac{x}{q^2} \right)dg.
\end{align*}
Shintani \cite{S75} studied the poles and residues of this function in the case $q = 1$.  The following is the analogue of \cite{S75} Lemma 4.
\begin{lemma}
Suppose $f$ vanishes on the singular forms in $L_r$.  We have
\begin{align*}
 Z_{q,r}(f,s_1, s_2) &= Z_{q,r}^+(f, s_1, s_2) + \hat{Z}_{q,r}^+\left(\hat{f},s_1, \frac{3}{2}-s_1-s_2\right)\\& + \frac{1}{8}\sum_{q_1q_2 = q} q_2^{-2} \prod_{p|q_1} p^{-s_1} \left(1 - \frac{1}{p^2}\right) \frac{\zeta(s_1)}{s_2-1} (\Phi_+ + \Phi_-)(f,s_1-1, 0) \\
 &+\frac{ \Sigma\left(\hat{f}, s_1-1\right)}{16s_1+16s_2 -24} \sum_{q_1q_2q_3 = q'} \prod_{p|q_1}\left(\frac{1}{p} + \frac{1}{p^2} - \frac{1}{p^3} \right)\prod_{p|q_2} \left(\frac{1}{p} - \frac{1}{p^3} \right)\prod_{p|q_3} \left(\frac{1}{p^3}-\frac{1}{p^4} \right)\\&\times\sum_{\substack{\ell \in \zed \setminus \{0\}\\ (\ell, q_2q_3)=1}} \sum_{\substack{b=1\\ (b,q_3)=1}}^\infty \frac{\omega_{2,q}(\ell, b)\phi(q_2b) q^{2s_1}}{(q_1^2q_2^3b^2|\ell|)^{s_1}}.
\end{align*}

\end{lemma}
\begin{proof}
 By Poisson summation (see \cite{S75}, Lemma 4)
 \begin{align*}
  Z_{q,r}(f,s_1,s_2) &= Z_{q,r}^+(f, s_1, s_2) + \hat{Z}_{q,r}^+\left(\hat{f}, s_1, \frac{3}{2}-s_1-s_2 \right)\\
  &+ \int_{B^+/B_{\zed}^+, \chi(g) \leq 1} \chi_1(g)^{s_1}\chi(g)^{s_2-\frac{3}{2}} \sum_{x \in \hat{L}_r \setminus \hat{L}_r'} \hat{\sN}_{q}^r(x)\hat{f}\left(\rho^*(g)\cdot \frac{x}{q^2}\right) dg. 
 \end{align*}
 First treat the singular forms in $\hat{L}_{0,r}(III)$, which recall has the decomposition
\begin{align*}
 \hat{L}_{0,r}(III) = \bigsqcup_{\ell \neq 0} \bigsqcup_{b>0} \bigsqcup_{\substack{0 \leq d < b\\ (b,d)=1}} \bigsqcup_{\gamma \in B_\zed^+} \{\gamma \cdot \ell(b^2, 2bd, d^2)\}.
\end{align*}
Let $\omega_{2,q}(\ell, b)$ have value 1 if $2\nmid q$ and have the values given if $2|q$,
\begin{align*}
 \omega_{2,q}(\ell,b) = \left\{\begin{array}{lll} \frac{5}{8} && 4|\ell\\ \frac{3}{8} && 2\|\ell, 2|b\\ \frac{1}{16} &&2\nmid\ell,b\\ 0 && \text{otherwise}\end{array}\right..
\end{align*}
Factor the odd part $q' = \frac{q}{(q,2)}$ of $q$ as $q_1q_2q_3$ where $q_1^2|\ell$, if $p|q_2$ then $p\|\ell$ and $q_2|b$ and $\gcd(q_3, \ell b)=1$.  Thus,
\begin{equation}
 \hat{\sN}_{q'}^r(\ell(b^2, 2bd, d^2)) = \prod_{p|q_1}\left(\frac{1}{p} +\frac{1}{p^2} - \frac{1}{p^3}\right) \prod_{p|q_2}\left(\frac{1}{p}-\frac{1}{p^3}\right) \prod_{p|q_3} \left(\frac{1}{p^3} - \frac{1}{p^4}\right).
\end{equation}
Thus the contribution from $\hat{L}_{0,r}(III)$ is
\begin{align*}
 &\sum_{q_1q_2q_3 = q'} \prod_{p|q_1}\left(\frac{1}{p} + \frac{1}{p^2} - \frac{1}{p^3} \right)\prod_{p|q_2} \left(\frac{1}{p} - \frac{1}{p^3} \right)\prod_{p|q_3} \left(\frac{1}{p^3}-\frac{1}{p^4} \right)\int_{B^+, \chi(g) \leq 1} \chi_1(g)^{s_1}\chi(g)^{s_2 - \frac{3}{2}}\\
&\times  \sum_{\substack{\ell \in \zed \setminus \{0\}\\ (\ell, q_2q_3)=1}} \sum_{\substack{ b= 1\\ (b, q_3)=1}}^\infty \omega_{2,q}(\ell,b)\sum_{\substack{0 \leq d < q_2b\\ (d, q_2b) = 1}}\hat{f}\left(\rho^*(g) \cdot \frac{q_1^2q_2\ell}{q^2}(q_2^2 b^2, 2q_2bd, d^2) \right) dg\\
&=\sum_{q_1q_2q_3 = q'} \prod_{p|q_1}\left(\frac{1}{p} + \frac{1}{p^2} - \frac{1}{p^3} \right)\prod_{p|q_2} \left(\frac{1}{p} - \frac{1}{p^3} \right)\prod_{p|q_3} \left(\frac{1}{p^3}-\frac{1}{p^4} \right)\int_0^1 \frac{d\lambda}{\lambda}\int_0^\infty \frac{dt}{t^3} \int_{-\infty}^\infty du\\
&\times (\lambda t)^{2s_1} \lambda^{4s_2 - 6}\sum_{\substack{\ell \in \zed \setminus \{0\}\\ (\ell, q_2q_3)=1}} \sum_{\substack{ b= 1\\ (b, q_3)=1}}^\infty \frac{\omega_{2,q}(\ell,b)\phi(q_2b)}{q_2b} \hat{f}\left(\frac{q_1^2q_2\ell}{q^2\lambda^2}\left(t^2 q_2^2 b^2, 2q_2b u, \frac{u^2}{t^2} \right) \right).
\end{align*}
After a change in variable in $u, t$, this is given by
\begin{align*}
 &\frac{1}{2} \sum_{q_1q_2q_3 = q'} \prod_{p|q_1}\left(\frac{1}{p} + \frac{1}{p^2} - \frac{1}{p^3} \right)\prod_{p|q_2} \left(\frac{1}{p} - \frac{1}{p^3} \right)\prod_{p|q_3} \left(\frac{1}{p^3}-\frac{1}{p^4} \right)\\
 &\times \int_0^1 \frac{d\lambda}{\lambda} \lambda^{2s_1 + 4s_2 - 6}\int_0^\infty \frac{dt}{t} t^{s_1-1} \int_{-\infty}^\infty du \sum_{\substack{\ell \in \zed \setminus\{0\}\\ (\ell, q_2q_3)=1}} \sum_{\substack{b = 1\\ (b,q_3)=1}}^\infty \frac{\omega_{2,q}(\ell, b) \phi(q_2b)}{(q_2b)^{s_1}}\hat{f}\left(\frac{q_1^2q_2^2 \ell b}{q^2 \lambda^2}\left(t, 2u, \frac{u^2}{t} \right) \right)\\
 &=\frac{1}{2}\sum_{q_1q_2q_3 = q'} \prod_{p|q_1}\left(\frac{1}{p} + \frac{1}{p^2} - \frac{1}{p^3} \right)\prod_{p|q_2} \left(\frac{1}{p} - \frac{1}{p^3} \right)\prod_{p|q_3} \left(\frac{1}{p^3}-\frac{1}{p^4} \right)\\
 &\times \int_0^1 \frac{d\lambda}{\lambda} \lambda^{4s_1 + 4s_2 - 6}\int_0^\infty \frac{dt}{t} t^{s_1-1} \int_{-\infty}^\infty du \sum_{\substack{\ell \in \zed \setminus\{0\}\\ (\ell, q_2q_3)=1}} \sum_{\substack{b = 1\\ (b,q_3)=1}}^\infty \frac{\omega_{2,q}(\ell, b) \phi(q_2b)q^{2s_1}}{(q_1^2q_2^3b^2|\ell|)^{s_1}}\hat{f}\left(\sgn(\ell)\left(t, 2u, \frac{u^2}{t} \right) \right)\\
 &=\frac{ \Sigma\left(\hat{f}, s_1-1\right)}{16s_1+16s_2 -24} \sum_{q_1q_2q_3 = q'} \prod_{p|q_1}\left(\frac{1}{p} + \frac{1}{p^2} - \frac{1}{p^3} \right)\prod_{p|q_2} \left(\frac{1}{p} - \frac{1}{p^3} \right)\prod_{p|q_3} \left(\frac{1}{p^3}-\frac{1}{p^4} \right)\\&\times\sum_{\substack{\ell \in \zed \setminus \{0\}\\ (\ell, q_2q_3)=1}} \sum_{\substack{b=1\\ (b,q_3)=1}}^\infty \frac{\omega_{2,q}(\ell, b)\phi(q_2b) q^{2s_1}}{(q_1^2q_2^3b^2|\ell|)^{s_1}}.
\end{align*}

It remains to handle the terms in $\hat{L}_{0,r}(I)$ and $\hat{L}_{0,r}(II)$.  In the case of $\hat{L}_{0,r}(II)$, let $\varpi_{2,q}(m,n)$ be 1 if $2 \nmid q$ and $\hat{\sN}_2^r(0, 2m,n)$ otherwise.  Recall that for $p > 2$, $$\hat{\sN}_p^r(0,2m,n) = \left\{\begin{array}{lll} \frac{1}{p} +\frac{1}{p^2} - \frac{1}{p^3} && p^2|m,n\\ \frac{1}{p} - \frac{1}{p^3} && p^2|m, p\|n\\ 0 && \text{otherwise}\end{array}\right..$$
The contribution of $\hat{L}_{0,r}(II)$ is
\begin{align*}
& \int_{B^+/B_{\zed}^+} \chi(g)^{s_1}\chi(g)^{s_2 -\frac{3}{2}} \sum_{m \neq 0} \sum_{0 \leq n < |2m|} \sum_{\gamma \in B_\zed^+} \hat{\sN}_q^r(0, 2m, n) \hat{f}\left(g\cdot \gamma \left(0, \frac{2m}{q^2}, \frac{n}{q^2}\right)\right) dg\\
&= \sum_{q_1q_2 = \frac{q}{(q,2)}} \frac{1}{q_2^2} \prod_{p|q_2} \left(\frac{1}{p}-\frac{1}{p^3}\right) \int_0^1 \lambda^{2s_1+4s_2-6} \frac{d\lambda}{\lambda} \int_0^\infty t^{2s_1-2}\frac{dt}{t} \int_{-\infty}^\infty du \\&\times\sum_{m \neq 0} \sum_{0 \leq n < |2m|q_1}\varpi_{2,q}(m,n) \hat{f}\left(0, \frac{2m}{\lambda^2}, \frac{n + 2mq_1u}{\lambda^2 q_1 t^2} \right)
\end{align*}
Integration in $u$ vanishes, since $f$ vanishes on the singular set. Hence, the contribution is zero. 

For $p \geq 2$, $$\hat{\sN}_p^r(0,0,\ell) = \left\{\begin{array}{lll} \frac{1}{p} + \frac{1}{p^2} - \frac{1}{p^3} && p^2 |\ell\\ \frac{1}{p} - \frac{1}{p^3} && p\| \ell\\ 0 && \text{otherwise} \end{array} \right..$$
Integrating away the parabolic direction obtains the contribution
\begin{align*}
 &\sum_{q_1q_2 = q} q_2^{-2} \prod_{p|q_1} \left(\frac{1}{p} - \frac{1}{p^3}\right)\int_0^1 \lambda^{2s_1 + 4s_2 - 6}\frac{d\lambda}{\lambda}\int_0^\infty t^{2s_1-2} \frac{dt}{t} \sum_{\ell \in \zed}\hat{f}\left(0,0, \frac{\ell}{t^2\lambda^2 q_1} \right)\\
 &= \sum_{q_1q_2 = q} q_2^{-2} \prod_{p|q_1} \left(1 - \frac{1}{p^2}\right) \int_0^1 \lambda^{2s_1 + 4s_2-4} \frac{d\lambda}{\lambda} \int_0^\infty t^{2s_1} \frac{dt}{t} \sum_{\ell \in \zed \setminus \{0\}} \int_{x_2, x_3 \in \bR}f(t^2 \lambda^2 q_1 \ell,x_2, x_3)dx_2dx_3\\
 &= \frac{1}{8}\sum_{q_1q_2 = q} q_2^{-2} \prod_{p|q_1} p^{-s_1} \left(1 - \frac{1}{p^2}\right) \frac{\zeta(s_1)}{s_2-1} (\Phi_+ + \Phi_-)(s_1-1, 0)  .
\end{align*}
\end{proof}

\begin{lemma}
 Let $f = bv^2w + cvw^2 + dw^3$ have $D = c^2 -4bd$ which is square.  Let $p$ be prime.  Then $f$ is non-maximal at $p$ if and only if either $p|b$ or $p|D$.
\end{lemma}

\begin{proof}
First suppose $p$ is odd.
 If $p|b$ then the form meets condition $\sM_1$.  If $p|D$ then the quadratic form $f' = bv^2 + cvw + dw^2$ has a double root modulo $p^2$, so that condition $\sM_2$ is met.  Conversely, either $p|b$ or $p|D$ if condition $\sM_1$ or $\sM_2$ is met. 
 
 Next if $p = 2$ and $p|b$ then $f$ satisfies condition $\sM_1$.  If $p|D$ then $f'$ has a double root modulo 2, and hence satisfies condition $\sM_2'$.  Conversely, if either $\sM_1$ or $\sM_2'$ is satisfied then $2|bD$.
\end{proof}

As in \cite{S75} let $A(m,n)$ be the number of distinct solutions to the congruence $x^2 \equiv n \bmod m$. Let, see \cite{S75} Lemma 7,
\begin{align*}
 Z_{q,r,\square}(f, s_1, s_2) &= \int_{B^+/B_{\zed}^+}\chi_1(g)^{s_1}\chi(g)^{s_2} \sum_{x \in L_r, \Disc(x) = \square} \sN_{q,r}(x)f(\rho(g)\cdot x) dg\\
 Z_{q,r,\square}(f,s_1,s_2) &= \frac{1}{4}\Xi_q(s_1, s_2)\Phi_+(f, s_1-1, s_2-1).
\end{align*}
By the previous lemma,
\begin{align*}
 \Xi_q(s_1, s_2) &= \frac{1}{2} \sum_{n,m=1}^\infty \frac{A(4m,n^2)}{m^{s_1}n^{2s_2}}\one(q|mn)\\
 &= \frac{1}{2} \sum_{q_1q_2 = q} \sum_{\substack{m,n=1\\ \GCD(m, q_2)=1}}^\infty \frac{A(4q_1m, q_2^2n^2)}{(q_1m)^{s_1}(q_2n)^{2s_2}}.
\end{align*}
\begin{lemma}
 Let $p >2$ be prime.  Let $j \geq 0$ and $\alpha \geq 1$ and let $p^j$ be the largest power of $p$ dividing $n$.  The number of solutions of $x^2 \equiv n^2 \bmod p^{\alpha}$ is
 \begin{align*}
  A(p^{2\alpha}, n^2) &= \left\{\begin{array}{lll}p^\alpha && j \geq \alpha\\ 2p^j && j < \alpha \end{array}\right.\\
  A(p^{2\alpha+1}, n^2) &= \left\{\begin{array}{lll} p^{\alpha} && j > \alpha\\ 2p^j && j \leq \alpha \end{array}\right..
 \end{align*}
For $\alpha \geq 2$, the number of solutions to $x^2 \equiv n^2 \bmod 2^{\alpha}$ is
\begin{align*}
 A(2^{2\alpha}, n^2) &= \left\{\begin{array}{lll}2^\alpha && j \geq \alpha-2\\ 2^{j+2} && j < \alpha -2  \end{array} \right.\\
 A(2^{2\alpha + 1}, n^2) &= \left\{ \begin{array}{lll}2^\alpha && j \geq \alpha\\ 2^{j+2} && j < \alpha \end{array} \right..
\end{align*}

\end{lemma}

\begin{lemma}\label{lem:Xi_q}
 Let $\RE(s_1) = \sigma_1, \RE(s_2) = \sigma_2$.  In the domain $\sigma_1 > 1$, $2\sigma_2 > 1$, $2\sigma_1 + 2\sigma_2 -1 > 1$, 
 \begin{align*}
  |\Xi_q(s_1, s_2)| \ll_\epsilon q^\epsilon\zeta(\sigma_1)^2\zeta(2\sigma_2)\zeta(2\sigma_1 + 2\sigma_2 -1) \sum_{q_0q_1q_2=q} \frac{1}{q_0^{2\sigma_1 + 2\sigma_2-1}q_1^{\sigma_1}q_2^{2\sigma_2}}.
 \end{align*}

\end{lemma}

\section{Twisted Shintani $\sL$-functions}\label{twisted_zeta_function_section}
As in the previous work \cite{H17}, we rely on automorphic twists of the zeta functions introduced by Shintani.

Define, for squarefree $q$,
\begin{align}
 \sL^+_q(\E_r, s) &= \sum_{m=1}^\infty \frac{1}{m^s}  \sum_{i=1}^{h(m)}  \frac{\sN_q(x_{i,m})\E_r(g_{i,m})}{|\Gamma(i,m)|}\\
 \notag \sL^-_q(\E_r, s) &= \sum_{m=1}^\infty \frac{1}{m^s} \sum_{i=1}^{h(-m)}\sN_q(x_{i,-m}) \E_r(g_{i,-m}).
\end{align}
As in the previous section, let $f_G$ be defined on $G^1$ by $f_G(g) = \exp\left(-\tr g^t g\right)$ and extend $f_G$ to $G^+$ independent of the determinant.  Let $f_D(x) \in C_c^\infty(\bR^+)$.  Define 
\begin{equation}
 f_{\pm}(g\cdot x_{\pm}) = f_G(g)f_D(\chi(g)).
\end{equation}
The twisted orbital integrals are given by
\begin{align}
 Z^{\pm}_q(f_{\pm}, \E_r, L; s) &= \int_{G^+/\Gamma} \chi(g)^s \E_r(g^{-1}) \sum_{x \in L}\sN_q(x) f_{\pm}(g\cdot x) dg.
 \end{align}

 \begin{lemma}\label{factorization_lemma}
  In $\RE(s)>1$, 
  \begin{equation}
   Z^{\pm}_q(f_{\pm},\E_r, L;s) = \frac{\sqrt{\pi}K_{\frac{z}{2}}(2)}{12} \sL^{\pm}_q(\E_r,s)\tilde{f}_D(s).
  \end{equation}

 \end{lemma}
\begin{proof}
 Calculate
 \begin{align*}
  Z^+_q(f_+, \E_r, L; s) &= \int_{G^+/\Gamma} \chi(g)^s \E_r(g^{-1}) \sum_{x \in L}\sN_q(x) f_+(g\cdot x) dg\\
  &=\int_{G^+/\Gamma} \chi(g)^s \E_r(g^{-1}) \sum_{m=1}^\infty \sum_{i=1}^{h(m)} \frac{\sN_q(x_{i,m})}{|\Gamma(i,m)|} \sum_{\gamma \in \Gamma} f_+(g\gamma g_{i,m} \cdot x_+) dg\\
  &= \int_{G^+} \chi(g)^s \E_r(g^{-1}) \sum_{m=1}^\infty \sum_{i=1}^{h(m)} \frac{\sN_q(x_{i,m})}{|\Gamma(i,m)|} f_G(g g_{i,m} ) f_D(\chi(g) m)dg\\
  &= \frac{\tilde{f}_D(s)}{12} \sum_{m=1}^\infty \frac{1}{m^s} \sum_{i=1}^{h(m)} \frac{\sN_q(x_{i,m})\E_r(g_{i,m})}{|\Gamma(i,m)|}\int_{G^1} f_G(g) t(g)^{1+z} dg\\
  &= \sqrt{\pi} K_{\frac{z}{2}}(2) \frac{\tilde{f}_D(s)}{12} \sL^+_q(\E_r, s).
 \end{align*}
The proof in the case $Z^-_q$ is similar.
\end{proof}

 Define
 \begin{align}
  Z^{\pm, +}_q(f_{\pm}, \E_r, L; s) &= \int_{G^+/\Gamma, \chi(g) \geq 1} \chi(g)^s \E_r(g^{-1}) \sum_{x \in L}\sN_q(x)  f_{\pm}(g\cdot x) dg\\
 \notag \hat{Z}^{\pm, +}_q(\hat{f}_{\pm}, \E_r, \hat{L}; 1-s) &= \int_{G^+/\Gamma, \chi(g) \geq 1} \chi(g)^{1-s} \E_r(g^{-1}) \sum_{x \in \hat{L} \setminus \hat{L}_0}  \hat{\sN}_q(x)\hat{f}_{\pm}\left(g \cdot \frac{x}{q^2}\right) dg\\
 \notag Z^{\pm, 0}_q(\hat{f}_{\pm}, \E_r, \hat{L}; s) &= \int_{G^+/\Gamma, \chi(g) \leq 1} \chi(g)^{s-1} \E_{r}(g^{-1}) \sum_{x \in \hat{L}_0}\hat{\sN}_q(x)  \hat{f}_{\pm}\left(g^\iota \cdot \frac{x}{q^2}\right) dg.
\end{align}
As before, the first two integrals are entire, due to the rapid decay of $f$ and $\hat{f}$.
The last integral is equal to 
\begin{align*}
 Z^{\pm, 0}_q\left(\hat{f}_{\pm}, \E_r, \hat{L}; s \right) &= \int_{G^+/\Gamma, \chi(g) \leq 1} \chi(g)^{s-1} \E_r(g^{-1}) \sum_{x \in \hat{L}_0}\hat{\sN}_q(x) \hat{f}_{\pm}\left(g^\iota \cdot \frac{x}{q^2}\right) dg\\
 &= \int_0^1\frac{d\lambda}{\lambda} \lambda^{12s-12}\int_{G^1/\Gamma} \E_{r}(g^{-1}) \sum_{x \in \hat{L}_0}\hat{\sN}_q(x) \hat{f}_{\pm}^{\lambda^{-3}}\left(g \cdot \frac{x}{q^2}\right) dg.
\end{align*}

 The orbital integral satisfies a split functional equation, which is a result of applying the Poisson summation formula.
\begin{lemma}\label{split_functional_equation}
 We have
 \begin{equation}
  Z^{\pm}_q(f_{\pm}, \E_r, L;s) = Z^{\pm, +}_q(f_{\pm}, \E_r, L; s) + \hat{Z}^{\pm, +}_q(\hat{f}_{\pm}, \E_r, \hat{L}; 1-s) + Z^{\pm, 0}_q(\hat{f}_{\pm}, \E_r, \hat{L}; s).
 \end{equation}

\end{lemma}
The main proposition to be proved in this section is as follows.

Write $f \sim g$ if $f-g$ is entire.
\begin{proposition}\label{pole_proposition}
 In the case that $f$ is supported on $V_-$, 
 \begin{align*}
  &Z_q^{-, 0}(f_-, \E_r, L;s) \sim \sqrt{\pi}K_{\frac{z}{2}}(2) \sum_{q_1q_2 = q} q_2^{-2}\prod_{p|q_1}(p^{-3}-p^{-5})\\&\times \Biggl[\frac{q_1^{\frac{2(1-z)}{3}}\zeta\left(\frac{1-z}{3} \right)2^{\frac{z-1}{6}}3^{-1}\pi^{\frac{1+2z}{6}}}{12s-11-z} \cos\left(\frac{\pi(1-z)}{6} \right)\frac{\Gamma\left(\frac{1-z}{3} \right)\Gamma\left(\frac{2+z}{6}\right)}{\Gamma\left(\frac{5+z}{6}\right)}\tilde{f}_D\left(\frac{11+z}{12}\right) \\& + \frac{\xi(z)}{\xi(1+z)}  
  \frac{q_1^{\frac{2(1+z)}{3}}\zeta\left(\frac{1+z}{3} \right)2^{\frac{-z-1}{6}}3^{-1}\pi^{\frac{1-2z}{6}}}{12s-11+z} \cos\left(\frac{\pi(1+z)}{6} \right)\frac{\Gamma\left(\frac{1+z}{3} \right)\Gamma\left(\frac{2-z}{6}\right)}{\Gamma\left(\frac{5-z}{6}\right)}\tilde{f}_D\left(\frac{11-z}{12}\right)\\
  &+\frac{\zeta(3+z)q_1^{-z}2^{\frac{-5-z}{2}}}{12s - 15-3z}\tilde{f}_D\left(\frac{5+z}{4} \right) +\frac{\xi(z)}{\xi(1+z)}\frac{\zeta(3-z)q_1^{z}2^{\frac{-5+z}{2}}}{12s - 15+3z}\tilde{f}_D\left(\frac{5-z}{4} \right)\Biggr].
 \end{align*}
In the case that $f$ is supported on $V_+$,
\begin{align*}
 & Z_q^{+,0}(f_+, \E_r, L;s) \sim \sqrt{\pi}K_{\frac{z}{2}}(2)\sum_{q_1q_2=q}q_2^{-2}\prod_{p|q_1}(p^{-3}-p^{-5}) \\&\times\Biggl[\frac{q_1^{\frac{2(1-z)}{3}}\zeta\left(\frac{1-z}{3}\right)3^{\frac{-z-7}{4}} 2^{\frac{z-1}{6}}\pi^{\frac{1+2z}{6}}}{12s-11-z} \cos\left(\frac{\pi(1-z)}{6} \right)\frac{\Gamma\left(\frac{1-z}{3} \right)\Gamma\left(\frac{2+z}{6}\right)}{\Gamma\left(\frac{5+z}{6}\right)}\tilde{f}_D\left(\frac{11+z}{12}\right) \\& + \frac{\xi(z)}{\xi(1+z)}  
  \frac{q_1^{\frac{2(1+z)}{3}}\zeta\left(\frac{1+z}{3}\right)3^{\frac{z-7}{4}}2^{\frac{-z-1}{6}}\pi^{\frac{1-2z}{6}}}{12s-11+z} \cos\left(\frac{\pi(1+z)}{6} \right)\frac{\Gamma\left(\frac{1+z}{3} \right)\Gamma\left(\frac{2-z}{6}\right)}{\Gamma\left(\frac{5-z}{6}\right)}\tilde{f}_D\left(\frac{11-z}{12}\right)\\
  &+\frac{\zeta(3+z)q_1^{-z}2^{\frac{-5-z}{2}}3^{\frac{1+z}{4}}}{12s - 15-3z}\tilde{f}_D\left(\frac{5+z}{4} \right) +\frac{\xi(z)}{\xi(1+z)}\frac{\zeta(3-z)q_1^z2^{\frac{-5+z}{2}}3^{\frac{1-z}{4}}}{12s - 15+3z}\tilde{f}_D\left(\frac{5-z}{4} \right)\Biggr].
 \end{align*}
\end{proposition}
\begin{proof} 
 This is a result of combining Lemmas \ref{Theta_n_lemma}, \ref{Theta_1_c_lemma}, and \ref{Theta_2_c_lemma} below.
\end{proof}

Combined with the factorization formula in Lemma \ref{factorization_lemma}, and the split functional equation in Lemma \ref{split_functional_equation}, this proves Theorem \ref{main_theorem}.

\subsection{The singular integral}
As in \cite{H19}, set 
\begin{equation}
 J_q\left(\hat{f}\right)(g) = \sum_{x \in \hat{L}_0}\hat{\sN}_q(x)  \hat{f}\left(g\cdot \frac{x}{q^2}\right).
\end{equation}

The following lemma is proved in \cite{H19}.
\begin{lemma}
 Suppose for some $A>4$ that $\hat{f}(x) \ll \frac{1}{(1+ \|x\|)^A}$.  Then $J_q\left(\hat{f}\right) \in C(G^1/\Gamma, A-6)$.
\end{lemma}

The object of interest is 
\begin{equation}
 \sI_q\left(\hat{f}, \E_r\right) = \int_{G^1/\Gamma} \E_r\left(g^{-1}\right) J_q\left(\hat{f}\right)(g) dg
\end{equation}
since
\begin{equation}
 Z^{\pm, 0}_q\left(\hat{f}_{\pm}, \E_r, \hat{L};s \right) = \int_0^1 \lambda^{12s-12}\sI_q\left(\hat{f}_{\pm}^{\lambda^{-3}} ,\E_r\right) \frac{d\lambda}{\lambda}.
\end{equation}

In order to gain convergence in later integrals, we reinterpret this as a limit of an integral against an incomplete Eisenstein series as in \cite{S72},
\begin{equation}
 \sI_q\left(\hat{f}, \E_r\right) = \frac{\xi(2)}{\psi(1)}\lim_{w \to 1^+} (w-1) \int_{G^1/\Gamma} \sE(\psi, w; g) \E_r\left(g^{-1}\right) J_q\left(\hat{f}\right)(g) dg.
\end{equation}

Write the singular forms in $\hat{L}_0$ as
\begin{align}
 \hat{L}_0 &= \{0\} \sqcup \bigsqcup_{m=1}^\infty \bigsqcup_{\gamma \in \Gamma/\Gamma\cap N} \{\gamma \cdot (0,0,0,m)\} \sqcup \bigsqcup_{m=1}^\infty \bigsqcup_{n\in \zed} \bigsqcup_{\gamma \in \Gamma/\Gamma\cap N} \{\gamma \cdot (0,0,3m,n)\}\\
 \notag &= \{0\} \sqcup L_0(I) \sqcup \hat{L}_0(II).
\end{align}
Write 
\begin{equation}
 \sI_q\left(\hat{f}, \E_r\right) = \Theta^{(0)}_q(\E_r) + \Theta^{(1)}_q(\E_r) + \Theta^{(2)}_q(\E_r)
\end{equation}
as the sum of three limits.   
Since $\E_r$ is mean 0 on $\Gamma \backslash G^1/K$, $\Theta^{(0)} = 0$.  In the remaining two pieces it is necessary to separate the contributions of the constant term of the Eisenstein series and the non-constant terms.

Write
\begin{align}
 \Theta^{(1)}_q(\E_r) &= \frac{\xi(2)}{\psi(1)}\lim_{w \to 1^+} (w-1)\int_{G^1/\Gamma}\sE(\psi, w; g) \E_r(g^t) \sum_{x \in L_0(I)} \hat{\sN}_q(x) \hat{f}\left(g\cdot \frac{x}{q^2} \right)dg\\
 \notag &= \frac{\xi(2)}{\psi(1)}\lim_{w \to 1^+} (w-1)\\ \notag &\times\int_{G^1/\Gamma \cap N}\sE(\psi, w; g) \E_r(g^t) \sum_{m=1}^\infty \hat{\sN}_q\left(0,0,0,m\right) \hat{f}\left(g \cdot \left(0,0,0, \frac{m}{q^2}\right)\right) dg.
\end{align}
Define
\begin{align}
 \Theta^{(1),c}_q(\E_r)&= \frac{\xi(2)}{\psi(1)}\lim_{w \to 1^+} (w-1)\\ \notag &\times\int_{G^1/\Gamma \cap N} \sE(\psi, w; g)\E_r^c(g^t) \sum_{m=1}^\infty \hat{\sN}_q\left(0,0,0,m\right) \hat{f}\left(g \cdot \left(0,0,0, \frac{m}{q^2}\right)\right) dg\\
 \notag \Theta^{(1),n}_q(\E_r)&= \frac{\xi(2)}{\psi(1)}\lim_{w \to 1^+} (w-1)\\ \notag &\times\int_{G^1/\Gamma \cap N} \sE(\psi, w; g)\E_r^n(g^t) \sum_{m=1}^\infty  \hat{\sN}_q\left(0,0,0,m\right) \hat{f}\left(g \cdot \left(0,0,0, \frac{m}{q^2}\right)\right) dg.
\end{align}
Similarly, 
\begin{align}
 \Theta^{(2)}_q(\E_r) &= \frac{\xi(2)}{\psi(1)}\lim_{w \to 1^+} (w-1)\\ \notag &\times\int_{G^1/\Gamma}\sE(\psi, w; g) \E_r(g^t) \sum_{x \in \hat{L}_0(II)} \hat{\sN}_q(x) \hat{f}\left(g\cdot \frac{x}{q^2} \right)dg\\
 \notag &= \frac{\xi(2)}{\psi(1)}\lim_{w \to 1^+} (w-1)\\ \notag &\times\int_{G^1/\Gamma \cap N}\sE(\psi, w; g) \E_r(g^t) \sum_{m=1}^\infty \sum_{n\in \zed} \hat{\sN}_q(0,0,3m,n) \hat{f}\left(g \cdot \left(0,0,\frac{3m}{q^2}, \frac{n}{q^2}\right)\right) dg
\end{align}
and
\begin{align}
 \Theta^{(2),c}_q(\E_r)&= \frac{\xi(2)}{\psi(1)}\lim_{w \to 1^+} (w-1)\\ \notag &\times\int_{G^1/\Gamma \cap N} \sE(\psi, w; g)\E_r^c(g^t) \sum_{m=1}^\infty \sum_{n\in\zed}\hat{\sN}_q(0,0,3m,n) \hat{f}\left(g \cdot \left(0,0,\frac{3m}{q^2}, \frac{n}{q^2}\right)\right) dg\\
 \notag \Theta^{(2),n}_q(\E_r)&= \frac{\xi(2)}{\psi(1)}\lim_{w \to 1^+} (w-1)\\ \notag &\times\int_{G^1/\Gamma \cap N}\sE(\psi, w; g) \E_r^n(g^t) \sum_{m=1}^\infty \sum_{n \in \zed}\hat{\sN}_q(0,0,3m,n) \hat{f}\left(g \cdot \left(0,0,\frac{3m}{q^2}, \frac{n}{q^2}\right)\right) dg.
\end{align}

Due to the exponential decay of $\E_r^n(g^t$) in the cusp, we may in fact obtain 
\begin{align}
 \Theta^{(1),n}_q(\E_r) &= \int_{G^1/\Gamma \cap N} \E_r^n(g^t) \sum_{m=1}^\infty  \hat{\sN}_q(0,0,0,m)\hat{f}\left(g \cdot \left(0,0,0, \frac{m}{q^2}\right)\right) dg\\ \notag
 \Theta^{(2),n}_q(\E_r)&=\int_{G^1/\Gamma \cap N} \E_r^n(g^t) \sum_{m=1}^\infty \sum_{n \in \zed}\hat{\sN}_q(0,0,3m,n) \hat{f}\left(g \cdot \left(0,0,\frac{3m}{q^2}, \frac{n}{q^2}\right)\right) dg.
\end{align}

\subsection{The non-constant term}

\begin{lemma}
 We have $\Theta_q^{(1),n}(\E_r) = 0$.
\end{lemma}
\begin{proof}
 This follows as in the proof of Lemma 12 of \cite{H19}.  
\end{proof}

Define 
\begin{equation}\hat{G}_{\lambda,q}(x) =\sum_{\ell, m=1}^\infty \frac{\eta_{\frac{z}{2}}(3\ell mq)}{\ell^{1+x}(3mq)^{1 + 3x}}.
\end{equation}

\begin{lemma}
  For $\epsilon>0$, the function $\hat{G}_{\lambda,q}(x)$ is bounded on $\{x: \RE(x) \geq \epsilon\}$ by a constant depending only upon $\epsilon$.
\end{lemma}
\begin{proof}
This follows since the divisor function $\eta_{\frac{z}{2}}(n)$ grows slower than any power of $n$.
\end{proof}

Let 
\begin{align}
 W_\lambda(w_1, w_2) &= \frac{1}{2}\frac{\Gamma(1-w_2)\cos\left(\frac{\pi}{2}(1-w_2) \right)}{(2\pi)^{\frac{1 + w_1 + w_2}{2}}} \tilde{K}_{\frac{z}{2}}\left(\frac{w_1 + 3w_2-1}{2} \right);\\
 \notag & \tilde{K}_\nu(s) = 2^{s-2}\Gamma\left(\frac{s+\nu}{2} \right)\Gamma\left(\frac{s-\nu}{2} \right).
\end{align}

\begin{lemma}
 $W_\lambda$ is holomorphic in $\RE(w_1 + 3w_2)>1$, $\RE(w_2)<1$.  Let $0 < \epsilon < \frac{1}{2}$.  For $\epsilon \leq \RE(w_2) \leq 1-\epsilon$,
 \begin{equation}
  \left|\Gamma(1-w_2)\cos\left(\frac{\pi}{2}(1-w_2)\right)\right| \ll |w_2|^{\frac{1}{2}- \RE(w_2)}.
 \end{equation}

\end{lemma}

\begin{proof}
 See \cite{H19}, Lemma 14.
\end{proof}

Define $\Sigma_1^{\pm}(f, z_1, z_2) = \int_0^\infty \int_0^\infty f(0,0,t, \pm u) t^{z_1-1}u^{z_2 - 1} dt du$.
\begin{lemma}
 If $f$ is Schwarz class, then $\Sigma_1^{\pm}(f, z_1, z_2)$ is holomorphic in $\RE(z_1), \RE(z_2) >0$.  In this domain, for $\sigma_1, \sigma_2 > 0$,
 \begin{equation}
  \left|\Sigma_1^{\pm}(f, \sigma_1 + it_1, \sigma_2 + it_2)\right| \ll_{\sigma_1, \sigma_2, A_1, A_2} \frac{1}{(1+|t_1|)^{A_1}(1 + |t_2|)^{A_2}}.
 \end{equation}
For $t>0$, if $f^t(x) = f(tx)$ then $\Sigma_1^{\pm}(f^t, z_1, z_2) = t^{-z_1 - z_2}\Sigma_1^{\pm}(f, z_1, z_2).$
\end{lemma}

\begin{proof}
 See \cite{H19}, Lemma 15.
\end{proof}

\begin{lemma}
We have
 \begin{align}
  \hat{\Theta}_q^{(2),n}(\E_r) &= \frac{4}{\xi(z+1)} \sum_{\epsilon = \pm}\sum_{q_1q_2=q} q_1^{-2}\prod_{p|q_2} \left(p^{-3}-p^{-5} \right)\\&\times \oiint_{\substack{\RE(w_1, w_2)\\ = (1, \frac{1}{2})}} q_2^{2(w_1 + w_2)} \Sigma_1^\epsilon\left(\hat{f}, w_1, w_2\right) W_\lambda(w_1, w_2) \hat{G}_{\lambda, q_2}\left(\frac{w_1 + w_2 - 1}{2} \right)dw_1 dw_2.
 \end{align}

\end{lemma}
\begin{proof}
 See \cite{H19} Lemma 16, where the details are the same in the case of an even form.
\end{proof}

\begin{lemma}\label{Theta_n_lemma}
We have the evaluation
\begin{align*}
 &Z_q^{\pm, 0, n}(\hat{f}, \E_r, \hat{L}; s) = \frac{4}{\xi(z+1)}\sum_{\epsilon = \pm} \sum_{q_1q_2 = q} q_1^{-2} \prod_{p|q_2}\left(p^{-3} - p^{-5}\right) \\&\times\oiint_{\substack{\RE(w_1,w_2)\\ = (1, \frac{1}{2})}}\frac{q_2^{2(w_1 + w_2)}}{12s-12+3w_1 + 3w_2}\Sigma^\epsilon\left(\hat{f}, w_1, w_2\right) W_\lambda(w_1, w_2) \hat{G}_{\lambda, q_2}\left(\frac{w_1+w_2-1}{2}\right) dw_1 dw_2
\end{align*}
and has holomorphic continuation to $\bC$.
\end{lemma}

\begin{proof}
See \cite{H19} Lemma 17.
\end{proof}

\subsection{The constant term} In this section we obtain the residues of the twisted zeta functions which arise from the constant term by modifying the method of Shintani \cite{S72}.

\begin{lemma}
 We have
 \begin{align}
 & \Theta^{(1),c}_q(\E_r) = \frac{1}{3}\sum_{q_1q_2 = q} q_2^{-2}\prod_{p|q_1}(p^{-3}-p^{-5})\\ \notag \times
 &\left[q_1^{\frac{2(1-z)}{3}}\zeta\left(\frac{1-z}{3} \right)\Sigma_2\left(\hat{f}, \frac{1-z}{3}\right) + \frac{\xi(z)}{\xi(z+1)}q_1^{\frac{2(1+z)}{3}}\zeta\left(\frac{1+z}{3} \right)\Sigma_2\left(\hat{f}, \frac{1+z}{3}\right)  \right].
 \end{align}

\end{lemma}
\begin{proof}
 Since $$
\hat{\mathscr N}_q(0,0,0,m)
=
\sum_{\substack{q_1q_2=q\\q_2^2\mid m}}
q_2^{-2}\prod_{p\mid q_1}(p^{-3}-p^{-5}).
$$

Replacing $m$ by $q_2^2m$, we obtain 
 \begin{align}
\notag \Theta_q^{(1),c}(\mathbf E_r)={}&\frac{\xi(2)}{\psi(1)}\lim_{w\to1^+}(w-1)\sum_{q_1q_2=q}q_2^{-2}\prod_{p\mid q_1}(p^{-3}-p^{-5})\\
&\times\int_{G^1/(\Gamma\cap N)}\mathscr E(\psi,w;g)\left(t^{1+z}+\frac{\xi(z)}{\xi(1+z)}t^{1-z}\right)\sum_{m=1}^\infty\hat f\left(g\cdot\left(0,0,0,\frac{m}{q_1^2}\right)\right)\,dg.
 \end{align}
Since $\hat f$ is invariant under $k_\theta$ on the left and $(0,0,0,m/q_1^2)$ is invariant under $n_u$, integration over $K$ may be eliminated, while integration over $n_u$ selects the constant term from $\mathscr E(\psi,w;g)$. For $1<x_0<w$, the integral is
\begin{align*}
\int_0^\infty
&\left(\oint_{\RE u=x_0}\frac{\psi(u)}{w-u}\left(t^{1+u}+\frac{\xi(u)}{\xi(1+u)}t^{1-u}\right)\,du\right)\\
&\times\left(t^{1+z}+\frac{\xi(z)}{\xi(1+z)}t^{1-z}\right)\sum_{m=1}^\infty\hat f\left(0,0,0,\frac{m}{q_1^2t^3}\right)\frac{dt}{t^3}.
\end{align*}
This expresses the integral as
\begin{align}
 &\sum_{m=1}^\infty \int_0^\infty \left(\oint_{\RE u = x_0} \frac{\psi(u)}{w-u} \left[t^{u+z} + \frac{\xi(u) t^{-u+z} }{\xi(u+1)}+ \frac{\xi(z)t^{u-z}}{\xi(z+1)} + \frac{\xi(u)\xi(z)t^{-u-z}}{\xi(u+1)\xi(z+1)} \right]du\right)\\\notag &\times\hat{f}\left(0,0,0, \frac{m}{q_1^2 t^3} \right)\frac{dt}{t}\\
 \notag =&\frac13\Bigg\{\oint_{\RE u=x_1}\frac{\psi(u)}{w-u}\left[q_1^{-\frac{2(u+z)}3}\zeta\left(-\frac{u+z}{3}\right)\Sigma_2\left(\hat f,-\frac{u+z}{3}\right)\right.\\
\notag&\hspace{3cm}\left.+\frac{\xi(z)}{\xi(1+z)}q_1^{-\frac{2(u-z)}3}\zeta\left(-\frac{u-z}{3}\right)\Sigma_2\left(\hat f,-\frac{u-z}{3}\right)\right]du\\\notag
&+\oint_{\RE u=x_2}\frac{\psi(u)}{w-u}\frac{\xi(u)}{\xi(1+u)}\left[q_1^{\frac{2(u-z)}3}\zeta\left(\frac{u-z}{3}\right)\Sigma_2\left(\hat f,\frac{u-z}{3}\right)\right.\\\notag
&\hspace{3cm}\left.+\frac{\xi(z)}{\xi(1+z)}q_1^{\frac{2(u+z)}3}\zeta\left(\frac{u+z}{3}\right)\Sigma_2\left(\hat f,\frac{u+z}{3}\right)\right]du\Bigg\}.
\end{align}
The first integral is holomorphic at $w=1$. The contribution of the second integral to the regularized limit comes from the pole at $u=1$ of $\xi(u)/\xi(1+u)$, whose residue is $1/\xi(2)$.

Thus the contribution at $w=1$ is
$$
\frac{\psi(1)}{3\xi(2)}\left[q_1^{\frac{2(1-z)}3}\zeta\left(\frac{1-z}{3}\right)\Sigma_2\left(\hat f,\frac{1-z}{3}\right)+
\frac{\xi(z)}{\xi(1+z)}q_1^{\frac{2(1+z)}3}\zeta\left(\frac{1+z}{3}\right)\Sigma_2\left(\hat f,\frac{1+z}{3}\right)\right].
$$

Multiplying by $\xi(2)/\psi(1)$ and summing over $q_1q_2=q$ gives the claim.

\end{proof}

\begin{lemma}\label{Theta_1_c_lemma}
 When $f$ is supported on $V_-$ the contribution to $Z^{\pm, 0}_q\left(\hat{f}, \E_r, \hat{L};s\right)$ from $\Theta^{(1),c}_q(\E_r)$ is
 \begin{align*}
&\sqrt{\pi}K_{\frac{z}{2}}(2)\sum_{q_1q_2 = q}q_2^{-2} \prod_{p|q_1} \left(p^{-3}-p^{-5} \right)\\&\times\Biggl[\frac{2^{\frac{z-1}{6}}3^{-1}\pi^{\frac{1+2z}{6}} q_1^{\frac{2(1-z)}{3}} \zeta\left(\frac{1-z}{3} \right)}{12s-11-z} \cos\left(\frac{\pi(1-z)}{6} \right)\frac{\Gamma\left(\frac{1-z}{3} \right)\Gamma\left(\frac{2+z}{6}\right)}{\Gamma\left(\frac{5+z}{6}\right)}\tilde{f}_D\left(\frac{11+z}{12}\right) \\& + \frac{\xi(z)}{\xi(1+z)}  
  \frac{2^{\frac{-z-1}{6}}3^{-1}\pi^{\frac{1-2z}{6}}  q_1^{\frac{2(1+z)}{3}}\zeta\left(\frac{1+z}{3} \right) }{12s-11+z} \cos\left(\frac{\pi(1+z)}{6} \right)\frac{\Gamma\left(\frac{1+z}{3} \right)\Gamma\left(\frac{2-z}{6}\right)}{\Gamma\left(\frac{5-z}{6}\right)}\tilde{f}_D\left(\frac{11-z}{12}\right) \Biggr].
 \end{align*}
When $f$ is supported on $V_+$, the contribution is
\begin{align*}
&\sqrt{\pi}K_{\frac{z}{2}}(2)\sum_{q_1q_2 = q}q_2^{-2} \prod_{p|q_1} \left(p^{-3}-p^{-5} \right)\\&\times\Biggl[\frac{ 2^{\frac{z-1}{6}}3^{\frac{-z-7}{4}}\pi^{\frac{1+2z}{6}} q_1^{\frac{2(1-z)}{3}}\zeta\left(\frac{1-z}{3}\right)}{12s-11-z} \cos\left(\frac{\pi(1-z)}{6} \right)\frac{\Gamma\left(\frac{1-z}{3} \right)\Gamma\left(\frac{2+z}{6}\right)}{\Gamma\left(\frac{5+z}{6}\right)}\tilde{f}_D\left(\frac{11+z}{12}\right) \\& + \frac{\xi(z)}{\xi(1+z)}  
  \frac{2^{\frac{-z-1}{6}}3^{\frac{z-7}{4}}\pi^{\frac{1-2z}{6}}q_1^{\frac{2(1+z)}{3}}\zeta\left(\frac{1+z}{3} \right) }{12s-11+z} \cos\left(\frac{\pi(1+z)}{6} \right)\frac{\Gamma\left(\frac{1+z}{3} \right)\Gamma\left(\frac{2-z}{6}\right)}{\Gamma\left(\frac{5-z}{6}\right)}\tilde{f}_D\left(\frac{11-z}{12}\right) \Biggr].
\end{align*}

\end{lemma}

\begin{proof}
We show only the negative discriminant case, since the positive is multiplied by a factor $3^{\frac{3\left(\frac{1\pm z}{3}\right)}{4}-1} = 3^{\frac{\pm z-3}{4}}$, see Lemma \ref{Sigma_2_lemma}.  The negative discriminant case is 
 \begin{align*}
  &\sum_{q_1q_2 = q}q_2^{-2} \prod_{p|q_1} \left(p^{-3}-p^{-5} \right)\int_0^1 \frac{d \lambda}{\lambda}\lambda^{12s-12} \frac{1}{3}\biggl[q_1^{\frac{2(1-z)}{3}}\zeta\left(\frac{1-z}{3}\right)\Sigma_2\left(\hat{f}^{\lambda^{-3}}, \frac{1-z}{3}\right) \\&+ \frac{\xi(z)}{\xi(1+z)}q_1^{\frac{2(1+z)}{3}}\zeta\left(\frac{1 + z}{3} \right)\Sigma_2\left(\hat{f}^{\lambda^{-3}}, \frac{1+z}{3} \right) \biggr].
  \end{align*}
  Since $
\Sigma_2(\hat f^{\lambda^{-3}},w)
=\lambda^{3w}\Sigma_2(\hat f,w)
$, 
\begin{align*}
\int_0^1\lambda^{12s-12}\Sigma_2\left(\hat f^{\lambda^{-3}},\frac{1-z}{3}\right)\frac{d\lambda}{\lambda}
&=\frac{\Sigma_2(\hat f,(1-z)/3)}{12s-11-z},\\
\int_0^1\lambda^{12s-12}\Sigma_2\left(\hat f^{\lambda^{-3}},\frac{1+z}{3}\right)\frac{d\lambda}{\lambda}
&=\frac{\Sigma_2(\hat f,(1+z)/3)}{12s-11+z}.
\end{align*}
Hence the contribution is
  \begin{align*}
  =&\sum_{q_1q_2 = q}q_2^{-2} \prod_{p|q_1} \left(p^{-3}-p^{-5} \right) \\&\times\frac{1}{3} \left[\frac{q_1^{\frac{2(1-z)}{3}}\zeta\left(\frac{1-z}{3} \right)\Sigma_2\left(\hat{f}, \frac{1-z}{3} \right)}{12s -11-z}  +\frac{\xi(z)}{\xi(1+z)}\frac{q_1^{\frac{2(1+z)}{3}}\zeta\left(\frac{1+z}{3} \right)\Sigma_2\left(\hat{f}, \frac{1+z}{3} \right)}{12s-11+z}   \right]\\
  =&\sum_{q_1q_2 = q}q_2^{-2} \prod_{p|q_1} \left(p^{-3}-p^{-5} \right)\\&\times\frac{1}{3}\Biggl[\frac{q_1^{\frac{2(1-z)}{3}}\zeta\left(\frac{1-z}{3} \right)2^{\frac{z-1}{6}}\pi^{\frac{2+z}{3}}}{12s-11-z} \cos\left(\frac{\pi(1-z)}{6} \right)\frac{\Gamma\left(\frac{1-z}{3} \right)\Gamma\left(\frac{2+z}{6}\right)}{\Gamma\left(\frac{5+z}{6}\right)}\tilde{f}_D\left(\frac{11+z}{12}\right) K_{\frac{z}{2}}(2)\\& + \frac{\xi(z)}{\xi(1+z)}  
  \frac{q_1^{\frac{2(1+z)}{3}}\zeta\left(\frac{1+z}{3} \right)2^{\frac{-z-1}{6}}\pi^{\frac{2-z}{3}}}{12s-11+z} \cos\left(\frac{\pi(1+z)}{6} \right)\frac{\Gamma\left(\frac{1+z}{3} \right)\Gamma\left(\frac{2-z}{6}\right)}{\Gamma\left(\frac{5-z}{6}\right)}\tilde{f}_D\left(\frac{11-z}{12}\right) K_{\frac{z}{2}}(2)\Biggr].
 \end{align*}

\end{proof}

This contributes the first set of poles in Proposition \ref{pole_proposition}.

\begin{lemma}
 We have
\begin{align}
&\Theta_q^{(2),c}(\E_r)=\frac{1}{2}\sum_{q_1q_2=q} q_2^{-2} \prod_{p|q_1} \left(p^{-1}-p^{-3} \right)\\
\notag &\left[q_1^{-2-z}\zeta(3+z) \int_{x = (x_2,x_3,x_4)}f_0(x) |x_2|^{2+z} dx+ \frac{\xi(z)}{\xi(1+z)}q_1^{-2+z}\zeta(3-z)\int_{x=(x_2, x_3, x_4)} f_0(x)|x_2|^{2-z}dx \right].
\end{align}

\end{lemma}

\begin{proof}
 Calculate
 \begin{align}
 &\notag \Theta^{(2),c}_q(\E_r) = \frac{\xi(2)}{\psi(1)} \lim_{w \to 1^+} (w-1) \int_{G^1/\Gamma\cap N} \sE(\psi, w;g)\\&\notag\times\left[t^{1+z} + \frac{\xi(z)}{\xi(z+1)}t^{1-z} \right]\sum_{m=1}^\infty \sum_{n \in \zed} \hat{\sN}_q(0,0,3m,n) \hat{f}\left(g \cdot \left(0,0, \frac{3m}{q^2}, \frac{n}{q^2}\right)\right)dg
 \end{align}
The integral may be written
\begin{align}
&\sum_{q_1q_2=q} q_2^{-2} \prod_{p|q_1} \left(p^{-3}-p^{-5} \right)\\&
\times \int_{G^1/\Gamma \cap N} \sE(\psi, w;g)\notag\left[t^{1+z} + \frac{\xi(z)}{\xi(z+1)}t^{1-z}\right] \sum_{m=1}^\infty \sum_{n \in \zed} \hat{f}\left(g \cdot \left(0,0,\frac{3m}{q_1}, \frac{n}{q_1^2} \right)\right)dg.
\end{align}
Using that $\hat{f}$ is left $K$ invariant, write this as 
\begin{align}
&\sum_{q_1q_2=q} q_2^{-2} \prod_{p|q_1} \left(p^{-3}-p^{-5} \right)\\& \notag \times \int_0^\infty \left[t^{-2 + z} + \frac{\xi(z)}{\xi(z+1)}t^{-2-z} \right]dt \int_0^1 du \\\notag &\times \sum_{m=1}^\infty \sum_{0 \leq n <3mq_1} \hat{f}\left(a_t \cdot \left(0,0, \frac{3m}{q_1}, \frac{3mq_1 u + n}{q_1^2}\right)\right) \oint_{\substack{\RE v = x_0\\ 4 < x_0 < \RE w}}\frac{E(v, a_t n_u) \psi(v)}{w-v} dv\\
 \notag &= \sum_{q_1q_2=q} q_2^{-2} \prod_{p|q_1} \left(p^{-3}-p^{-5} \right)\\&\notag\times
 \int_0^\infty \left[t^{-2 + z} + \frac{\xi(z)}{\xi(z+1)}t^{-2-z} \right]dt \int_{-\infty}^\infty du \\\notag &\times \sum_{m=1}^\infty  \hat{f}\left(a_t \cdot \left(0,0, \frac{3m}{q_1}, \frac{u}{q_1^2}\right)\right)\frac{1}{3mq_1}\sum_{0 \leq n <3mq_1} \oint_{\substack{\RE v = x_0\\ 4 < x_0 < \RE w}}\frac{E(v, a_t n_{\frac{u-n}{3mq_1}}) \psi(v)}{w-v} dv.
\end{align}
The sum over $n$ selects the Fourier coefficients of $E(v, \cdot)$ which are divisible by $3mq_1$.  Split the remaining coefficients into the constant term and the non-constant term.  The constant term contributes
\begin{align}
 &\sum_{q_1q_2=q} q_2^{-2} \prod_{p|q_1} \left(p^{-3}-p^{-5} \right)\\\notag&\times \int_0^\infty \left[t^{-2 + z} + \frac{\xi(z)}{\xi(z+1)}t^{-2-z}  \right]dt \int_{-\infty}^\infty du\\
 \notag &\times \sum_{m=1}^\infty \hat{f}\left(0,0, \frac{3m}{ q_1 t}, \frac{u}{ q_1^2 t^3} \right) \oint_{\substack{\RE v = x_0\\ 4 < x_0 < \RE w}} \frac{t^{1+v} + \frac{\xi(v)}{\xi(v+1)} t^{1-v}}{w-v} \psi(v) dv   .
\end{align}
Split the last contour integral into two parts corresponding to $t^{1+v}$ and $t^{1-v}$.  In the $t^{1+v}$ part, shift the integral leftward arbitrarily to show that this has no singularity at $w=1$.  This reduces to the $t^{1-v}$ part, which we write as
\begin{align}
 &\sum_{q_1q_2=q} q_2^{-2} \prod_{p|q_1} \left(p^{-1}-p^{-3} \right)\\&\notag \times \oint_{\substack{\RE v = x_0\\ 4 < x_0 < \RE w}} dv \frac{\xi(v)\psi(v)}{\xi(v+1)(w-v)} \int_0^\infty\left[t^{3 + z-v} +\frac{\xi(z)}{\xi(z+1)} t^{3 -z-v} \right]\frac{dt}{t}\\
 \notag &\times \int_{-\infty}^\infty du\sum_{m=1}^\infty \hat{f}\left(0,0,\frac{3m}{q_1t},u \right)
\\ \notag
&=\sum_{q_1q_2=q} q_2^{-2} \prod_{p|q_1} \left(p^{-1}-p^{-3} \right)\\
\notag & \times \oint_{\substack{\RE v = x_0\\ 4 < x_0 < \RE w}} dv\frac{\xi(v)\psi(v)}{\xi(v+1)(w-v)} \\&\times \notag \biggl[\zeta(v-3-z) \left(\frac{q_1}{3} \right)^{v-3-z}\Sigma_3(\hat{f},v-3-z)\\&\notag + \frac{\xi(z)}{\xi(z+1)} \zeta(v-3+z)\left(\frac{q_1}{3} \right)^{v-3+z}\Sigma_3(\hat{f},v-3+z) \biggr].\end{align}
The contribution to $\Theta^{(2),c}(\E_r)$ comes from the pole at $v=1$ and obtains
\begin{align}&\sum_{q_1q_2=q} q_2^{-2} \prod_{p|q_1} \left(p^{-1}-p^{-3} \right)\\
\notag &\left[\zeta(-2-z)\left(\frac{q_1}{3} \right)^{-2-z}\Sigma_3(\hat{f}, -2-z) + \frac{\xi(z)}{\xi(z+1)} \zeta(-2 + z)\left(\frac{q_1}{3} \right)^{-2+z} \Sigma_3(\hat{f}, -2 + z) \right].
\end{align}
Combine this with Lemma \ref{Sigma_3_lemma} and the functional equation of the Riemann zeta function to obtain the claimed quantity in the lemma,
\begin{align}
&\frac{1}{2}\sum_{q_1q_2=q} q_2^{-2} \prod_{p|q_1} \left(p^{-1}-p^{-3} \right)\\
\notag &\left[q_1^{-2-z}\zeta(3+z) \int_{x = (x_2,x_3,x_4)}f_0(x) |x_2|^{2+z} dx+ \frac{\xi(z)}{\xi(1+z)}q_1^{-2+z}\zeta(3-z)\int_{x=(x_2, x_3, x_4)} f_0(x)|x_2|^{2-z}dx \right].
\end{align}

The contribution of the non-constant terms of the Fourier series for $E(v,\cdot)$ is
\begin{align} &\sum_{q_1q_2=q} q_2^{-2} \prod_{p|q_1} \left(p^{-1}-p^{-3} \right)\\
 \notag&\times \int_0^\infty dt \int_{-\infty}^\infty du \oint_{\substack{\RE v = x_0\\ 4 < x_0 < w}} dv \frac{4}{\xi(v+1)} \frac{\psi(v)}{w-v}\\
 \notag &\times \left[t^{2+z} + \frac{\xi(z)}{\xi(z+1)} t^{2-z} \right] \sum_{\ell,m=1}^\infty \eta_{\frac{v}{2}} (3\ell m q_1)K_{\frac{v}{2}}(6\pi \ell m q_1 t^2) \cos(2\pi \ell t^3 q_1^2 u) \hat{f}\left(0,0,\frac{3m}{q_1 t}, u\right).
\end{align}
Since the Fourier transform $\hat{f}$ is Schwarz class, and the $K$ Bessel function has exponential decay in large variable, the $v$ integral may be passed to the left of the 1 line, which proves that this term is holomorphic at $w=1$, hence does not contribute to $\Theta^{(2),c}(\E_r).$
\end{proof}

\begin{lemma}\label{Theta_2_c_lemma}
 In the case that $f$ is supported on $V_-$, $\Theta^{(2),c}_q$ contributes  
 \begin{align*}
 &\sqrt{\pi} K_{\frac{z}{2}}(2)\sum_{q_1q_2=q} q_2^{-2}\prod_{p|q_1}(p^{-1}-p^{-3})\\&\times \left[\frac{\zeta(3+z)q_1^{-2-z}2^{\frac{-3-z}{2}}}{12s - 15-3z}\tilde{f}_D\left(\frac{5+z}{4} \right) +\frac{\xi(z)}{\xi(1+z)}\frac{\zeta(3-z)q_1^{-2+z}2^{\frac{-3+z}{2}}}{12s - 15+3z}\tilde{f}_D\left(\frac{5-z}{4} \right)\right].
 \end{align*}
to $Z^{\pm, 0}\left(\hat{f}, \E_r, \hat{L};s\right)$.
 In the case that $f$ is supported on $V_+$, the contribution is
\begin{align*}
 &\sqrt{\pi} K_{\frac{z}{2}}(2)\sum_{q_1q_2=q} q_2^{-2}\prod_{p|q_1}(p^{-1}-p^{-3})\\ &\times\left[\frac{\zeta(3+z)q_1^{-2-z}2^{\frac{-3-z}{2}}3^{\frac{1+z}{4}}}{12s - 15-3z}\tilde{f}_D\left(\frac{5+z}{4} \right) +\frac{\xi(z)}{\xi(1+z)}\frac{\zeta(3-z)q_1^{-2+z}2^{\frac{-3+z}{2}}3^{\frac{1-z}{4}}}{12s - 15+3z}\tilde{f}_D\left(\frac{5-z}{4} \right)\right].
\end{align*}

\end{lemma}
\begin{proof}
Apply Lemma \ref{Sigma_3_lemma} to find
 \begin{align*}
 & \sum_{q_1q_2=q} q_2^{-2}\prod_{p|q_1} (p^{-1}-p^{-3})\int_0^1 \frac{d\lambda}{\lambda}\lambda^{12s-12}\\&\times \Biggl[\zeta(-2-z) \left(\frac{3}{q_1}\right)^{2+z} \Sigma_3\left(\hat{f}^{\lambda^{-3}},-2-z\right)+ \frac{\xi(z)}{\xi(1+z)}\zeta(-2+z)\left(\frac{3}{q_1}\right)^{2-z} \Sigma_3\left(\hat{f}^{\lambda^{-3}}, -2+z\right)\Biggr]\\
 &=\sum_{q_1q_2=q} q_2^{-2}\prod_{p|q_1} (p^{-1}-p^{-3})\\
 &\times \Biggl[\frac{\zeta(-2-z)\left(\frac{3}{q_1}\right)^{2+z}\Sigma_3\left(\hat{f}, -2-z\right)}{12s-15-3z} + \frac{\xi(z)}{\xi(1+z)}\frac{\zeta(-2+z)\left(\frac{3}{q_1}\right)^{2-z}\Sigma_3\left(\hat{f}, -2+z\right)}{12s-15+3z}\Biggr]\\
 &= \frac{1}{2}\sum_{q_1q_2=q} q_2^{-2}\prod_{p|q_1} (p^{-1}-p^{-3})\\
 &\times\left[\frac{\zeta(3+z)q_1^{-2-z}}{12s-15-3z}\int f_0(x)|x_2|^{2+z} dx +\frac{\xi(z)}{\xi(1+z)}\frac{\zeta(3-z)q_1^{-2+z}}{12s-15+3z}\int f_0(x)|x_2|^{2-z} dx \right].
 \end{align*}
In the case of $V_-$, by Lemma \ref{Phi_0_lemma} this is
\begin{align*}
 &\sqrt{\pi} K_{\frac{z}{2}}(2)\sum_{q_1q_2=q}q_2^{-2}\prod_{p|q_1}(p^{-1}-p^{-3})\\&\times\left[\frac{\zeta(3+z)q_1^{-2-z}2^{\frac{-3-z}{2}}}{12s - 15-3z}\tilde{f}_D\left(\frac{5+z}{4} \right) +\frac{\xi(z)}{\xi(1+z)}\frac{\zeta(3-z)q_1^{-2+z}2^{\frac{-3+z}{2}}}{12s - 15+3z}\tilde{f}_D\left(\frac{5-z}{4} \right)\right].
\end{align*}
In the case of $V_+$, this is
\begin{align*}
& \sqrt{\pi} K_{\frac{z}{2}}(2)\sum_{q_1q_2=q}q_2^{-2}\prod_{p|q_1}(p^{-1}-p^{-3})\\&\times\left[\frac{\zeta(3+z)q_1^{-2-z}2^{\frac{-3-z}{2}}3^{\frac{1+z}{4}}}{12s - 15-3z}\tilde{f}_D\left(\frac{5+z}{4} \right) +\frac{\xi(z)}{\xi(1+z)}\frac{\zeta(3-z)q_1^{-2+z}2^{\frac{-3+z}{2}}3^{\frac{1-z}{4}}}{12s - 15+3z}\tilde{f}_D\left(\frac{5-z}{4} \right)\right].
\end{align*}
\end{proof}

\section{Reducible twisted Shintani $\sL$-functions}

Introduce reducible versions of the zeta functions as follows, the superscript $r$ indicating sums are restricted to reducible forms,
\begin{equation}
 Z^{\pm, r}_q(f, \E_r, L;s) = \int_{G^+/\Gamma} \chi(g)^s \E_r(g^{-1}) {\sum_{x \in L }}^r \sN_q(x) f(g\cdot x) dg
\end{equation}
and
\begin{align}
 \sL^{+, r}_q(\E_r, s) &= \sum_{m=1}^\infty \frac{1}{m^s}  {\sum_{i=1}^{h(m), r}}  \frac{\sN_q(x_{i,m})\E_r(g_{i,m})}{|\Gamma(i,m)|}\\
 \notag
 \sL^{-,r}_q(\E_r, s) &= \sum_{m=1}^\infty \frac{1}{m^s} {\sum_{i=1}^{h(-m), r}}  \sN_q(x_{i,-m})\E_r(g_{i,-m}).
\end{align}
As in the previous Section, with the same choice of $f_G$ and $f_D$, and with the same proof
\begin{equation}
  Z^{\pm,r}_q(f_{\pm},\E_r, L;s) = \frac{\sqrt{\pi}K_{\frac{z}{2}}(2)}{12} \sL^{\pm,r}_q(\E_r,s)\tilde{f}_D(s).
\end{equation}.

Recall that in the description of the region $\sR$ which is a fundamental domain for reducible forms modulo $\Gamma$  (see \cite{S75} pp.\ 45-46), $a = 0$, the first non-zero coefficient $b$ is assumed to be positive and $0 \leq c < 2|b|$.    Dropping the restriction $b>0$ and introducing a factor of $\frac{1}{2}$ to compensate, then using the $K$-invariance in the last line, \begin{align}
 Z^{\pm, r}_q(f_{\pm}, \E_r, L;s) &= \frac{1}{2}\int_{G^+/\Gamma} \chi(g)^s \E_r(g^t) \sum_{\gamma \in \Gamma/\Gamma\cap N}{\sum_{x =(0,b,c,d) \in \zed^4  }} \frac{\sN_q(x)f_{\pm}(g\gamma \cdot x)}{1 + 2 \one(c^2 -4bd = \square)}  dg \\
 &\notag = \frac{1}{2}\int_{G^+/\Gamma \cap N} \chi(g)^s \E_{r}(g^t) \sum_{x = (0,b,c,d)\in \zed^4} \frac{\sN_q(x)f_{\pm}(g \cdot x)}{1 + 2 \one(c^2 -4bd = \square)}dg\\
 &\notag = \frac{1}{2}\int_{B^+/B^+_\zed } \chi(g)^s \E_{r}(g^t) \sum_{x = (0,b,c,d)\in \zed^4} \frac{\sN_q(x)f_{\pm}(g \cdot x)}{1 + 2 \one(c^2 -4bd = \square)}dg.
\end{align}
Note that in making these manipulations, we find a representation $x = (0,b,c,d)$ for a given class of reducible forms, but this representative need not have corresponding group element $g_x$ corresponding to the choice $g_{i,m}$ representing the class, for instance, $(g_x^{-1})^t$ may not lie in a Siegel set, although this will be the case in the crucial situation of interest corresponding to quadratic fields.  The choice of representative does not matter when evaluating the Eisenstein series $\E_r$, which is $\Gamma$-invariant, but matters when making the pairing with the constant or non-constant terms $\E_r^c, \E_r^n$, so we have to track the dependence.

Set 
\begin{align}
 Z_{1,q}^{\pm, r}(f_{\pm}, \E_r, L; s) &= \frac{1}{2}\int_{B^+/B_\zed^+}\chi(g)^s \E_r(g^t)\sum_{x = (0,b,c,d)} \sN_q(x) f_{\pm}(g\cdot x)dg,\\ \notag
 Z_{\square, q}^{\pm, r}(f_{\pm}, \E_r, L;s) &= \frac{1}{3} \int_{B^+/B_{\zed}^+} \chi(g)^s \E_r(g^t) \sum_{\substack{x = (0,b, c, d) \in \zed^4\\ c^2 - 4bd = \square}} \sN_q(x)f_{\pm}(g \cdot x) dg.
\end{align}
Thus $$Z^{\pm, r}_q(f_{\pm}, \E_r, L; s) = Z_{1,q}^{\pm, r}(f_{\pm}, \E_r, L; s) - Z_{\square,q}^{\pm, r}(f_{\pm}, \E_r, L;s).$$
Separate the Eisenstein series into its constant term and non-constant term, indicating the constant term parts with a subscript $c$, e.g.
\begin{align}
  Z_{1,c,q}^{\pm, r}(f_{\pm}, \E_r, L; s)& = \frac{1}{2}\int_{B^+/B_\zed^+}\chi(g)^s \E_r^c(g^t)\sum_{x = (0,b,c,d)} \sN_q(x) f_{\pm}(g\cdot x)dg\\
\notag
Z_{\square,c,q}^{\pm, r}(f_{\pm}, \E_r, L; s)& = \frac{1}{3}\int_{B^+/B_\zed^+}\chi(g)^s \E_r^c(g^t)\sum_{\substack{x = (0,b,c,d)\\ c^2 - 4bd = \square}} \sN_q(x) f_{\pm}(g\cdot x)dg\\
\notag Z_{c,q}^{\pm, r}(f_{\pm}, \E_r, L; s) &= Z_{1,c,q}^{\pm, r}(f_{\pm}, \E_r, L; s) - Z_{\square, c, q}^{\pm, r}(f_{\pm}, \E_r, L; s)
\end{align}

In this work we will avoid the analysis of the non-constant part of the Eisenstein series on reducible forms, instead just proving that is makes a small contribution.  The following lemma collects the representation of the constant term contribution.
\begin{lemma}\label{orbital_zeta_constant_term_reducible}
 The orbital zeta function for the reducible terms and constant term of the Eisenstein series has the representation
 \begin{align}
  Z_{1,c, q}^{\pm, r}(f_{\pm}, \E_r, L; s) &= \frac{1}{12} \sum_{\substack{x=(0,b,c,d) \in \zed^4\\ 0 \leq c < 2b\\ \pm \Disc(x) >0 }} \frac{ \sN_q(x)\E^c_r\odot f_G(g_x)}{|\Disc(x)|^s}\tilde{f}_D(s),\\
\notag Z_{\square,c, q}^{\pm, r}(f_{\pm}, \E_r, L; s) &= \frac{1}{18} \sum_{\substack{x=(0,b,c,d) \in \zed^4\\ 0 \leq c < 2b\\ \pm \Disc(x) >0\\ c^2-4bd = \square }} \frac{ \sN_q(x)\E^c_r\odot f_G(g_x)}{|\Disc(x)|^s}\tilde{f}_D(s),\\
 Z_{c, q}^{\pm, r}(f_{\pm}, \E_r, L; s) &= \frac{1}{12} \sum_{\substack{x=(0,b,c,d) \in \zed^4\\ 0 \leq c < 2b\\ \pm \Disc(x) >0 }} \frac{ \sN_q(x)\E^c_r\odot f_G(g_x)}{|\Disc(x)|^s(1 + 2\one(c^2 - 4bd = \square))}\tilde{f}_D(s) \notag \\
\notag Z_{q}^{\pm, r}(f_{\pm}, \E_r, L; s) &=\frac{\sqrt{\pi} K_{\frac{z}{2}}(2)}{12}  \sum_{\substack{x=(0,b,c,d) \in \zed^4\\ 0 \leq c < 2b\\ \pm \Disc(x) >0 }} \frac{\sN_q(x)\E_r(g_x)}{|\Disc(x)|^s(1 + 2 \one (c^2-4bd = \square))} \tilde{f}_D(s).
 \end{align}

\end{lemma}
\begin{proof}
 The results directly follow from the definition of convolution.
\end{proof}

Indicate the action of $g \in \GL_2$ on binary cubic forms by $g_3 \cdot x$ and the action on binary quadratic forms by $g_2 \cdot x$. 
For $x = (b,c,d) \in \zed^3$, define $f_0(x) = f(0,x)$.  Write the integral as
\begin{align*}
 &Z_{1,q}^{\pm, r}(f_{\pm}, \E_r, L;s)\\ &= \frac{1}{2}\int_0^\infty \frac{d\lambda}{\lambda} \int_0^\infty \frac{dt}{t^3} \int_0^1 du \lambda^{12s} \E_{r}(n_u^{t}a_t) \sum_{x = (b,c,d)}\sN_q^r(x) f_{\pm,0}\left(\left(d_{\sqrt{\frac{\lambda^3}{t}}}a_t n_u\right)_2 \cdot x\right)\\
 &= \frac{1}{3}\int_0^\infty \frac{d\lambda}{\lambda} \int_0^\infty \frac{dt}{t^3} \int_0^1 du \lambda^{8s}t^{4s} \E_{r}(n_u^{t}a_t) \sum_{x = (b,c,d)}\sN_q^r(x)f_{\pm,0}\left(\left( d_{\lambda} a_t n_u\right)_2 \cdot x\right).
\end{align*}
 The constant term part is, 
\begin{align*}
 Z_{1,c, q}^{\pm, r} &= \frac{1}{3} \int_0^\infty \frac{d\lambda}{\lambda} \int_0^\infty \frac{dt}{t^3} \int_0^1 du \lambda^{8s}\left(t^{4s+1+z} + \frac{\xi(z)}{\xi(1+z)}t^{4s+1-z} \right)\sum_{x = (b,c,d)}\sN_q^r(x) f_{\pm,0}((d_\lambda a_t n_u)_2 \cdot x). 
\end{align*}

\begin{lemma} \label{lemma:pole}
 When $f$ is supported on $V_-$, the constant term function $Z_{1,c,q}^{\pm, r}$ has a pair of simple poles at $\frac{5}{4}+\frac{z}{4}$ and $\frac{5}{4}-\frac{z}{4}$, with residues
 \begin{align*}
  \frac{1}{24} \zeta(3+ z)2^{\frac{-1- z}{2}} \prod_{p|q}\left(\frac{1}{p^2}+\frac{1}{p^{3+z}} - \frac{1}{p^{5+z}} \right)\tilde{f}_D\left(\frac{5+ z}{4} \right)\sqrt{\pi}K_{\frac{z}{2}}(2), \qquad &s = \frac{5}{4} + \frac{z}{4},\\
  \frac{\xi(z)}{\xi(1+z)}\frac{1}{24} \zeta(3- z)2^{\frac{-1+ z}{2}}\prod_{p|q}\left(\frac{1}{p^2}+\frac{1}{p^{3-z}} - \frac{1}{p^{5-z}} \right) \tilde{f}_D\left(\frac{5- z}{4} \right)\sqrt{\pi}K_{\frac{z}{2}}(2), \qquad &s = \frac{5}{4} - \frac{z}{4}.
 \end{align*}
 When $f$ is supported on $V_+$, the residues are 
  \begin{align*}
  \frac{1}{24} \zeta(3+ z)2^{\frac{-1- z}{2}}3^{\frac{1+ z}{4}} \prod_{p|q}\left(\frac{1}{p^2}+\frac{1}{p^{3+z}} - \frac{1}{p^{5+z}} \right)\tilde{f}_D\left(\frac{5+ z}{4} \right)\sqrt{\pi}K_{\frac{z}{2}}(2), \qquad &s = \frac{5}{4} + \frac{z}{4},\\
  \frac{\xi(z)}{\xi(1+z)}\frac{1}{24} \zeta(3- z)2^{\frac{-1+ z}{2}}3^{\frac{1- z}{4}}\prod_{p|q}\left(\frac{1}{p^2}+\frac{1}{p^{3-z}} - \frac{1}{p^{5-z}} \right) \tilde{f}_D\left(\frac{5- z}{4} \right)\sqrt{\pi}K_{\frac{z}{2}}(2), \qquad &s = \frac{5}{4} - \frac{z}{4}.
 \end{align*}
Besides these two poles, the function is holomorphic in $\RE(s)>\frac{3}{4}$ and has the representation there
\begin{align*}
Z_{1,c,q}^{\pm, r} &=  \frac{1}{3}\left(Z_{q,r}\left(f_{\pm,0}, 2s+\frac{1}{2} + \frac{z}{2}, s - \frac{1}{4}-\frac{z}{4}\right) + \frac{\xi(z)}{\xi(1+z)}Z_{q,r}\left(f_{\pm,0}, 2s+\frac{1}{2}-\frac{z}{2}, s - \frac{1}{4} + \frac{z}{4}\right) \right).
\end{align*}

\end{lemma}

\begin{proof}
 Recall that, with $Z(f, s_1, s_2)$ the orbital zeta function of the space of binary quadratic forms as in Shintani \cite{S75},
 \begin{align*}
   Z_{1,c,q}^{\pm, r} &= \frac{1}{3} \int_0^\infty \frac{d\lambda}{\lambda} \int_0^\infty \frac{dt}{t^3} \int_0^1 du \lambda^{8s}\left(t^{4s+1+z} + \frac{\xi(z)}{\xi(1+z)}t^{4s+1-z} \right)\sum_{x = (b,c,d)}\sN_q^r(x) f_{\pm,0}((d_\lambda a_t n_u)_2 \cdot x)\\
 &= \frac{1}{3}\left(Z_{q,r}\left(f_{\pm,0}, 2s+\frac{1}{2} + \frac{z}{2}, s - \frac{1}{4}-\frac{z}{4}\right) + \frac{\xi(z)}{\xi(1+z)}Z_{q,r}\left(f_{\pm,0}, 2s+\frac{1}{2}-\frac{z}{2}, s - \frac{1}{4} + \frac{z}{4}\right) \right). 
 \end{align*}
Recall 
\begin{align*}
 Z_{q,r}(f,s_1, s_2) &= Z_{q,r}^+(f, s_1, s_2) + \hat{Z}_{q,r}^+\left(\hat{f},s_1, \frac{3}{2}-s_1-s_2\right)\\& + \frac{1}{8}\sum_{q_1q_2 = q} q_2^{-2} \prod_{p|q_1} p^{-s_1} \left(1 - \frac{1}{p^2}\right) \frac{\zeta(s_1)}{s_2-1} (\Phi_+ + \Phi_-)(f,s_1-1, 0) \\
 &+\frac{ \Sigma\left(\hat{f}, s_1-1\right)}{16s_1+16s_2 -24} \sum_{q_1q_2q_3 = q'} \prod_{p|q_1}\left(\frac{1}{p} + \frac{1}{p^2} - \frac{1}{p^3} \right)\prod_{p|q_2} \left(\frac{1}{p} - \frac{1}{p^3} \right)\prod_{p|q_3} \left(\frac{1}{p^3}-\frac{1}{p^4} \right)\\&\times\sum_{\substack{\ell \in \zed \setminus \{0\}\\ (\ell, q_2q_3)=1}} \sum_{\substack{b=1\\ (b,q_3)=1}}^\infty \frac{\omega_{2,q}(\ell, b)\phi(q_2b) q^{2s_1}}{(q_1^2q_2^3b^2|\ell|)^{s_1}}.
\end{align*}
Substituting $s_1 = 2s + \frac{1}{2} +\frac{z}{2}$, $s_2 = s-\frac{1}{4}-\frac{z}{4}$, all but the second line is holomorphic in $\RE(s)> \frac{3}{4}$.  The second line has a pole at $s = \frac{5+z}{4}$ where $s_2=1$ which obtains the claimed simple pole.  The second term is similar.

To obtain the claimed formulae, note that Lemma \ref{Phi_0_lemma} establishes that
\begin{align*}
 \Phi_-\left(f_{+,0}, 2\pm z, 0\right) &= 2^{\frac{-1\mp z}{2}}\tilde{f}_D\left(\frac{5\pm z}{4} \right)\sqrt{\pi} K_{\frac{z}{2}}(2),\\
 \Phi_+\left(f_{-,0}, 2\pm z, 0\right) &= 3^{\frac{1\pm z}{4}}2^{\frac{-1\mp z}{2}}\tilde{f}_D\left(\frac{5\pm z}{4} \right)\sqrt{\pi} K_{\frac{z}{2}}(2).
\end{align*}

\end{proof}
The poles and residues in this lemma at $s = \frac{5\pm z}{4}$ match those from Lemma \ref{Theta_2_c_lemma}.

\begin{lemma}
 The sum over square discriminants,
 $$
  Z_{q,c,\square}^{\pm, r} = \frac{2}{9} \int_0^\infty \frac{d\lambda}{\lambda} \int_0^\infty \frac{dt}{t^3} \int_0^1 du \lambda^{8s} \left(t^{4s+1+z} + \frac{\xi(z)}{\xi(1+z)}t^{4s+1-z}\right) \sum_{\substack{x = (b,c,d)\\ \Disc(x) = \square}}\sN_q^r(x) f_{\pm,0}((d_\lambda a_t n_u)_2 \cdot x)
 $$
is holomorphic in $\RE(s) > \frac{3}{4}$.
\end{lemma}
\begin{proof}
We have,
$$
Z_{q,c,\square}^{\pm, r} =\frac{2}{9} Z_{q,\square}\left(f, 2s + \frac{1}{2} + \frac{z}{2}, s-\frac{1}{4} - \frac{z}{4}\right) +\frac{2}{9} \frac{\xi(z)}{\xi(1+z)}Z_{q,\square}\left(f, 2s + \frac{1}{2} -\frac{z}{2}, s - \frac{1}{4} +\frac{z}{4}\right),
$$
where 
$$
 Z_{q,\square}(f, s_1, s_2) = \frac{1}{4}\Xi_q(s_1, s_2) \Phi_+(f, s_1-1, s_2-1).
$$
By restricting the support of $f$ away from the singular set, assume $\Phi_+(f, s_1-1, s_2-1)$ is entire.  We have
$$
 \Xi_q(s_1, s_2) = \frac{1}{2}\sum_{q_1q_2=q} \sum_{\substack{m,n = 1\\\GCD(m, q_2)=1}}^\infty \frac{A(4q_1m,q_2^2n^2)}{(q_1m)^{s_1}(q_2n)^{2s_2}}
$$
is holomorphic in $\RE(s)> \frac{3}{4}$.

\end{proof}

 Combining the above results proves that the reducible orbital zeta function is meromorphic in $\RE(s)> \frac{3}{4}$, with simple poles at $\frac{5\pm z}{4}$.

\section{The sieve, and Proof of Theorem \ref{main_theorem}}

Define
 $$
  N_{3,\pm}'(\E_r, F, X) = \sum_{m \in \zed \setminus \{0\}} F\left(\frac{\pm m}{X}\right) {\sum_{i = 1}^{h(m)}}^* \frac{\E_r(g_{i,m})}{|\Gamma(i,m)|}
 $$
 where the $*$ restricts summation to forms which are irreducible and maximal at all primes $p$.
 By inverse Mellin transform and M\"{o}bius inversion sifting to maximal orders,
$$
 N_{3,\pm}'(\E_r, F, X) = \sum_{q} \mu(q) \oint_{\RE(s) = 4} \hat{F}(s) X^s \left(\sL_q^{\pm}(\E_r, s) - \sL_q^{\pm, r}(\E_r, s) \right)ds.
$$
Apply $x \frac{d}{dx}$ twice to $F$ to obtain $f_D$, whose Mellin transform then satisfies $s^2 \tilde{F}(s) = \tilde{f}_D(s)$.  Recalling the expressions for $\sL$ in terms of the orbital zeta functions, we obtain
$$
 N_{3,\pm}'(\E_r, F, X) =\frac{12}{\sqrt{\pi}K_{\frac{z}{2}}(2)}\sum_{q} \mu(q) \oint_{\RE(s) = 4} X^s \left(Z_q^{\pm}(f,\E_r, s) - Z_q^{\pm, r}(f,\E_r, s) \right)\frac{ds}{s^2}
$$
 
 \begin{lemma}
  The Weyl sums from Theorem \ref{main_theorem} satisfy 
  $$
   S^\pm(\E_r, F, X) = \frac{1}{2} N_{3, \pm}'(\E_r, F, X) + O_\epsilon\left(X^{\frac{7}{12}+\epsilon} \right)
  $$
as $X \to \infty$.
 \end{lemma}
\begin{proof}
 Those forms which are maximal at all primes $p$ and irreducible correspond to maximal orders in cubic fields.  There are $O(X^{\frac{1}{2}})$ cyclic cubics.  For a cyclic cubic represented by form $g_{i,m} \cdot x_{\pm}$, the value of $t$ in the Iwasawa decomposition is bounded by $X^{\frac{1}{12}}$.  Since the constant term of the Eisenstein series grows as $\ll t$, the cyclic cubics contribute $O_\epsilon\left(X^{\frac{7}{12}+\epsilon}\right)$.  The remaining cubic fields get counted with weight 2, which proves the lemma.
\end{proof}

We make a further adjustment, replacing the orbital zeta function in case of reducible forms with the version in which the pairing is only with the constant term of the Eisenstein series, making an error which will be estimated.  This gives
$$
 N_{3,\pm}''(\E_r, F, X) =\frac{12}{\sqrt{\pi}K_{\frac{z}{2}}(2)}\sum_{q} \mu(q) \oint_{\RE(s) = 4} X^s \left(Z_q^{\pm}(f,\E_r, s) - Z_{1,c,q}^{\pm, r}(f,\E_r, s) \right)\frac{ds}{s^2}
$$
The first stage in bounding this error is the following Lemma.

\begin{lemma}
The difference $|N_{3,\pm}'(\E_r, F, X) - N_{3,\pm}''(\E_r, F, X)| = O(1)$.
\end{lemma}

\begin{proof}
The change affects only reducible forms.  Since we sieve to maximal forms still, these are in bijective correspondence with fields of degree less than 3, with an appropriate weight.  The quadratic fields are represented in $\sR$ by a form either $x=(0, 1, 0, \frac{-D}{4})$ or $x=(0,1,1, \frac{-(D-1)}{4})$ with trivial stabilizer, and these have element $g_x$ with $(g_x^{-1})^t$ having $t$ with $t$ of order $D^{\frac{1}{4}}$.  It follows by Lemma \ref{non_constant_term_bound_convolution} that $\E_r \odot f_G (g_x) - \E_r^c \odot f_G(g_x) = \E_r^n \odot f_G(g_x) = O_A(D^{-A})$ so these forms contribute $O(1)$.  The contribution of $\bQ$ is also $O(1)$.
\end{proof}

\begin{lemma}
 The tail of the sieve satisfies the bound
 $$
  \sum_{q > Q} \mu(q)\oint_{\RE(s)=4} X^s (Z_q^{\pm}(f, \E_r,s) - Z_{1,c,q}^{\pm, r}(f, \E_r, s))\frac{ds}{s^2} \ll_r \frac{X^{\frac{13}{12}}}{Q^{1-\epsilon}} + X^{\frac{3}{4}+\epsilon}.
 $$
\end{lemma}

\begin{proof}
    Let $x = (a,b,c,d)$ be a binary cubib form and let $g_x=d_\lambda n_u a_t k_\theta$ satisfy $g_x\cdot x_\pm = x$. If $a\ne 0$, then $a\ge 1$, and the coordinate formula for the group action implies $t\gg X^{-\frac{1}{12}}$. If $a=0$ but $b\ne 0$, the same coordinate formula gives $\theta=0$, thus $t\gg X^{-\frac{1}{4}}$. Note that the Eisenstein series satisfies an estimate $$\E_r(g_x)=\E_r((g_x^{-1})^t)\ll_r t^{-1}.$$ 
    
    We handle separately the cases $x$ is irreducible and $x$ is reducible with either squre or non-square discriminant.

    \ 
    
    Case 1: Irreducible forms.
    For a generic form with $a\ne 0$, $$\E_r(g_x)\ll_r X^{\frac{1}{12}}$$ and there is no contribution from $Z_{1,c,q}^{\pm,r}$; the term only involves reducible forms. After grouping forms by discriminant size, up to $X$, the number satisfying $q^2|\Disc (x)$ is majorized by $X^{1+\epsilon} q^{-2+\epsilon}$. Summing over $q>Q$ gives $$\sum_{q>Q} q^{-2+\epsilon}\ll_\epsilon Q^{-1+\epsilon}. $$ Hence the tail estimate over the irreducible forms is majorized by $$X^{1/12}X^{1+\epsilon}\sum_{q>Q}q^{-2+\epsilon}\ll_{r,\epsilon} \frac{X^{13/12+\epsilon}}{Q^{1-\epsilon}}.$$

    Case 2: Reducible forms with nonsquare discriminant.
    Now let $x$ be reducible with nonsquare discriminant. Furthermore,  let $x_0$ be its one representative of the class in $\sR$.  Since the action on forms is by $\Gamma$, we may choose a representative for the class $x_1$ such that $(g_{x_1}^{-1})^t$ is in the standard fundamental domain for $\Gamma \backslash G^1$.  If $x_1$ has $a \neq 0$ then $t_{x_1} \gg X^{-\frac{1}{12}}$ which implies that $\E_r(g_x) = O(X^{\frac{1}{12}})$ in the class. Since the Eisenstein series is automorphic, replacing $x_0$ by an equivalent representative does not change its value on the associate shape. Hence the nonconstant contribution satisfies $$(\E_r - \E_r^c) \odot f_g(g_{x_0}) \ll_r X^{\frac{1}{12}}.$$ Using the same method from Case 1 gives the same order of magnitude. If instead $x_1$ has $a=0$, then $b\ne 0$. In this case the smoothed nonconstant Eisenstein contribution is bounded independently of $X$: $(\E_r - \E_r^c) \odot f_G(g_{x_1}) = O(1)$, which gives smaller terms than $\frac{X^{13/12+\epsilon}}{Q^{1-\epsilon}}$.

    Case 3: Reducible forms with square discriminant. There are either one, say $x_1$ or three say $x_1, x_2, x_3$ representatives for the form $x$ in $\sR$. For these forms, we use the easy bound $t\gg X^{-1/4}$ so $|\E_r(g_{x_i})|\ll_r t^{-1}\ll_r X^{1/4}.$ Using the same upper bound for $\E_r^c$, we obtain $$|\E_r(g_{x_i})-\E_r^c(g_{x_i})|\ll X^{1/4},$$ which in turn gives $$(\E_r-\E_r^c)\odot f_G(g_{x_i})\ll_r X^{1/4}, \quad \E_r^c\odot f_G(g_{x_i})\ll_r X^{1/4}.$$
    In this case the discriminant of the reducible forms $\Disc (f) =b^2(c^2-4bd)$ is square. Letting $m=b$ and $n^2=c^2-4bd$, we get $\Disc (f)=m^2n^2$. From Lemma \ref{lem:Xi_q}, using $s_1=2s, s_2=s$ and $\sigma_1=2\sigma, \sigma_2=\sigma$, we get $$|\Xi_q(2s,s)|\ll\zeta(2\sigma)^3\zeta(6\sigma-1)\sum_{q_0q_1q_2=q}\frac{1}{q_0^{6\sigma-1}q_1^{2\sigma}q_2^{2\sigma}}.$$ Choose $\sigma=\frac{1}{2}+\delta$ for $\delta>0$. Then $$|\Xi_q(2s,s)|\ll_\delta \sum_{q_0q_1q_2=q}\frac{1}{q_0^{2+6\delta}q_1^{1+2\delta}q_2^{1+2\delta}}=\prod_{p|q}\left(\frac{1}{p^{2+6\delta}}+\frac{2}{p^{1+2\delta}}\right).$$ Since $p^{-2-6\delta}\le p^{-1-2\delta}$, $$\prod_{p|q}\left(\frac{1}{p^{2+6\delta}}+\frac{2}{p^{1+2\delta}}\right)\le \frac{3^{\omega(q)}}{q^{1+2\delta}}.$$ Since $3^{\omega(q)}\ll_q q^\delta$, we obtain $$|\Xi_q(2s,s)|\ll\delta q^{-1-\delta}.$$ Hence the number of square discriminant less than $X$ with $q^2$ dividing the discriminant can be obtained as $$X^{\sigma+\epsilon}|\Xi_q(2\sigma,\sigma)|\ll_\delta \frac{X^{1/2+\delta+\epsilon}}{q^{1+\delta}}$$ for any $\sigma=1/2+\delta>1/2$. Summing over $q>Q$, $$X^{1/2+\delta+\epsilon}\sum_{q>Q}\frac{1}{q^{1+\delta}}\ll_{\delta,\epsilon} \frac{X^{1/2+\delta+\epsilon}}{Q^\delta}.$$ Hence the tail estimate over the squre discriminant forms is majoirzed by $$X^{1/4}\frac{X^{1/2+\delta+\epsilon}}{Q^\delta}=\frac{X^{3/4+\delta+\epsilon}}{Q^\delta}\ll_{r,\delta,\epsilon} X^{3/4+\delta+\epsilon},$$ and letting $\delta\to 0$, this yields $$X^{3/4+\epsilon}.$$
    
    Combining everything, we obtain the desired result.
\end{proof}

Dropping the tail of the sieve leaves
$$
  N_{3,\pm}'''(\E_r, F, X) =\frac{12}{\sqrt{\pi}K_{\frac{z}{2}}(2)}\sum_{q<Q} \mu(q) \oint_{\RE(s) = 4} X^s \left(Z_q^{\pm}(f,\E_r, s) - Z_{1,c,q}^{\pm, r}(f,\E_r, s) \right)\frac{ds}{s^2}
$$
We now open the orbital zeta functions, splitting e.g. $Z_q^{\pm}$ into its constant term and non-constant term parts, and expressing $Z_q^{\pm}, Z_{1,c,q}^{\pm, r}$ in the approximate functional equation.

Recall that we have the approximate functional equation representation $Z_q^{\pm} = Z_q^{\pm, +} + \hat{Z}_q^{\pm, +} + Z_q^{\pm, 0}$.  We now put in bounds for the individual terms.
\begin{lemma}\label{Z_bound}
 We have the bound 
 $$
  \sum_{q < Q}\mu(q) \oint_{\RE(s)=4} X^s Z_q^{\pm,+}(f, \E_r, s) \frac{ds}{s^2} = O(X^\epsilon).
 $$

\end{lemma}
\begin{proof}
 Recall that $\Phi_q(x) \neq 0$ implies $q|\Disc(x)$. We open the definition of $Z_q^{\pm, +}(f, \E_r, s)$ and shift the contour to $\RE(s) = \epsilon$ to obtain the bound
 \begin{align}
  &\sum_{q < Q}\mu(q) \oint_{\RE(s) = \epsilon}X^s \int_{G^+/\Gamma, \chi(g) \geq 1}\chi(g)^s \E_r(g^{-1}) \sum_{x \in L}\Phi_q(x)f(g\cdot x)dg \frac{ds}{s^2}\\
  &\notag
  =\sum_{q < Q}\mu(q) \oint_{\RE(s) = \epsilon} X^s \sum_m \sum_{i = 1}^{h(\pm m)} \frac{\Phi_q(x_{i, \pm m})}{|\Gamma(i, \pm m)|}\int_{G^+, \chi(g) \geq 1} \chi(g)^s \E_r(g^{-1})f(g\cdot x_{i, \pm m}) dg \frac{ds}{s^2}. 
 \end{align}
The compact support of $f_D$ makes the sums over $q$ and $m$ sums of $O(1)$ terms, and the inner integral is also $O(1)$ as a result.
\end{proof}

\begin{lemma}\label{hat_Z_bound}
 We have the bound 
 $$
  \sum_{q < Q}\mu(q) \oint_{\RE(s)=4} X^s \hat{Z}_q^{\pm,+}(f, \E_r, s) \frac{ds}{s^2} = O(Q^{4 + \epsilon}X^\epsilon).
 $$

\end{lemma}

\begin{proof}
 We again open the definition of $\hat{Z}_q^{\pm, +}$ and shift the contour to $\RE(s) = \epsilon$ obtaining the bound
 \begin{align}
  &\sum_{q < Q}\mu(q) \oint_{\RE(s) = \epsilon}X^s\int_{G^+/\Gamma, \chi(g)\geq 1}\chi(g)^{1-s} \E_r(g^{-1})\sum_{x \in \hat{L} \setminus \hat{L}_0}\hat{\Phi}_q(x)\hat{f}\left(g \cdot \frac{x}{q^2}\right) dg \frac{ds}{s^2}\\
  &\notag =\sum_{q < Q}\mu(q)\oint_{\RE(s) = \epsilon} X^s \sum_{m \neq 0}\sum_{i = 1}^{\hat{h}(m)}\frac{\hat{\Phi}_q(\hat{x}_{i,m})}{|\hat{\Gamma}(i,m)|} \int_{G^+, \chi(g)>1} \chi(g)^{1-s}\E_r(g^{t})\hat{f}(g \cdot \hat{x}_{i,m}) dg \frac{ds}{s^2}\\
  &\notag= \sum_{q < Q} \mu(q) \oint_{\RE(s) = \epsilon}X^s q^{8(1-s)} \sum_{m \neq 0}\sum_{i=1}^{\hat{h}(m)} \frac{\hat{\Phi}_q(\hat{x}_{i,m})}{|\hat{\Gamma}(i,m)| |m|^{1-s}}\\&\notag \times\int_{G^+, \chi(g) \geq \frac{|m|}{q^8}}\chi(g)^{1-s} \E_r((\hat{g}_{i,m}^{-1})^t g^t) f(g \cdot x_{\sgn m}) dg \frac{ds}{s^2}.
 \end{align}
The rapid decay of $\hat{f}$ limits the sum to $m \ll q^8 X^\epsilon$.  For those sums included, bound the inner integral over $G^1$ against $\E_r$ by $|\Disc(x_{i,m})|^{\frac{1}{4}} \ll \frac{q^2}{m^{\frac{1}{4}}}$.  Now insert the bound $\sum_{|m| \leq Y} \sum_{i = 1}^{\hat{h}(m)}\frac{|\hat{\Phi}(\hat{x}_{i,m})|}{|\hat{\Gamma}(i,m)|} \ll Y q^{-7+\epsilon}$.  Putting together these bounds recovers the estimate $\ll Q^{4 + \epsilon}X^\epsilon$.
\end{proof}

The part of the singular terms attached to the non-constant part of the Eisenstein series have the following bound.
\begin{lemma}We have the bound
  $$
  \sum_{q < Q}\mu(q) \oint_{\RE(s)=4} X^s Z_q^{\pm, 0, n}(f, \E_r, s) \frac{ds}{s^2} = O(X^{\frac{1}{4}}(QX)^\epsilon).
 $$

\end{lemma}
\begin{proof}
 See Lemma 22 of \cite{H19}, where the proof is the same.
\end{proof}
The part of the singular terms attached to the constant part of the Eisenstein series are polar terms, and we will handle these all together with the other polar terms from $Z_{1,c,q}^{\pm,r}$ after treating the remaining terms.

Recall the expression $Z_{1,c,q}^{\pm,r} = \frac{1}{3}\left(Z_{q,r}\left(f_{\pm, 0}, 2s + \frac{1}{2} + \frac{z}{2}, s-\frac{1}{4}-\frac{z}{4} \right) + \frac{\xi(z)}{\xi(1+z)}Z_{q,r}\left(f_{\pm, 0}, 2s + \frac{1}{2} -\frac{z}{2}, s-\frac{1}{4} + \frac{z}{4} \right)\right).$

\begin{lemma} We have the bound
$$
 \sum_{q < Q} \mu(q) \oint_{\RE(s) = 4} X^s Z_{q,r}^+\left(f_{\pm, 0}, 2s + \frac{1}{2} + \frac{z}{2}, s-\frac{1}{4} - \frac{z}{4}\right)\frac{ds}{s^2} = O\left(X^{\frac{1}{4}+\epsilon}\right).
$$

\end{lemma}
\begin{proof}
 Open the definition of $Z_{q,r}^+$ so that we need to estimate
 \begin{align*}
& \sum_{q < Q} \mu(q)\oint_{\RE(s) = 4} X^s\int_1^\infty \frac{d\lambda}{\lambda} \int_0^\infty \frac{dt}{t^3} \int_{-\infty}^\infty du \lambda^{8s}t^{4s+1+z}\\&\times \sum_{\substack{x = (b,c,d)\\0 \leq c < 2|b| \\ c^2-4bd \neq 0}}\sN_q^r(x) f_{\pm,0}\left(\lambda^2\left(t^2 b,c+2ub ,\frac{d+ cu + bu^2}{t^2} \right)\right)\frac{ds}{s^2}\\
&= \sum_{q < Q} \mu(q)\oint_{\RE(s) = 4} X^s \int_1^\infty \frac{d\lambda}{\lambda} \int_0^\infty \frac{dt}{t} \int_{-\infty}^\infty du \lambda^{8s}t^{4s-1+z}\\&\times \sum_{\substack{x = (b,c,d)\\0 \leq c < 2|b| \\ c^2-4bd \neq 0}}\frac{1}{b} \sN_q^r(x) f_{\pm, 0}\left(\lambda^2 \left(t^2 b, 2u, \frac{bd- \frac{c^2}{4} + u^2}{t^2b}\right)\right)\frac{ds}{s^2},
 \end{align*}
after the change of variables replacing $c+2ub$ by $2u$.
We have $\sN_q^r$ vanishes unless $q|b^2(c^2-4bd)$, so for every $\eta>0$, $$\sum_{q<Q} |\mu(q)|\sN_q^r(b,c,d)\ll_\eta |b|^{2\eta}|c^2-4bd|^\eta,$$ which is independent of $Q$.  Using the rapid decay of $f_{\pm, 0}$, which is compact support in the discriminate of the associated binary cubic form $\lambda^8t^4 b^2(c^2-4bd)$ and Schwarz class in the discriminant $\lambda^4(c^2-4bd)$ we can truncate to $\lambda = O(X^\epsilon)$, $t^2b = O(X^\epsilon)$ and $c^2-4bd = O(X^\epsilon)$.  Now push the contour to $\RE(s) = \frac{1}{4}+\epsilon$. This can be done due to the convergence and decay of the integrand in vertical strips, and no pole is crossed in this shift. Thus we gain the factor $X^{1/4+\epsilon}$, provided the the remaining integrals and sums contribute only $X^\epsilon$.

Write $q = q_1 q_2$ with $q_1$ the largest factor of $q$ that divides $c^2 - 4bd$, so that $q_2|b$.  For fixed $b,c,d$, the number of possible $q$'s satisfying these divisibility condition is bounded by $\tau(b(c^2-4bd))\ll_\epsilon X^{\epsilon}$, i.e., the sum over $q<Q$ contributes only $\epsilon$ power of $X$.

To see the convergence of the sum over $b$, make the change of variables $(t')^2 = t^2b$, so that the original sum becomes \begin{align*}&\sum_{q<Q}\mu(q)\int_{\RE(s)=4}X^s\int_1^\infty \frac{d\lambda}{\lambda}\int_0^\infty \frac{dt'}{t'} \int_{-\infty}^\infty du\,\lambda^{8s}t'^{4s-1+z}\\&
\times \sum_{\substack{x=(b,c,d)\\0\le c<2|b|\\c^2-4bd\ne 0}}\sN_q^4(x)b^{-2s-1/2-z/2}f_{\pm,0}\left(\lambda^2\left(t'^2,2u,\frac{bd-c^2/4+u^2}{t'^2}\right)\right)\frac{ds}{s^2},\end{align*} so now factor of $b$ now appears as $\frac{1}{b^{2s+\frac{1}{2}+4 \epsilon}}$. On the shifted contour $\sigma=1/4+\epsilon$, this becomes $|b|^{-1-2\epsilon}$. Given  $c$ there are $O(X^\epsilon)$ choices for $b, d$, so that, bounding $2|b|>c$ and summing in $c$ we obtain the desired bound.  
\end{proof}

Remark: if we had split the integral at $\lambda \geq Y$ instead of $\lambda \geq 1$ in the approximate functional equation this trades off the length of this sum with the length of the dual sum.  In principle this could be done so that the two give an equal error term.

\begin{lemma}
 We have the bound
 $$
  \sum_{q < Q}\mu(q) \oint_{\RE(s) = 4} X^s \hat{Z}_{q,r}^+\left(\hat{f}_{\pm, 0}, 2s + \frac{1}{2} + \frac{z}{2}, \frac{5}{4} - 3s - \frac{z}{4}\right)  \frac{ds}{s^2} = O_\epsilon(X^{\frac{1}{4}+\epsilon}Q^{2+8\epsilon}).
 $$

\end{lemma}
\begin{proof}
 We again open the definition of $\hat{Z}_{q,r}^+$, so that we need to bound
 \begin{align*}
  &\sum_{q < Q}\mu(q) \oint_{\RE(s) = 4} X^s \int_1^\infty \frac{d\lambda}{\lambda} \int_0^\infty \frac{dt}{t^3} \int_{-\infty}^\infty du \lambda^{6-8s}t^{4s + 1 + z}\\&\times \sum_{\substack{x = (b,c,d)\\ 0 \leq c < 2|b|\\ c^2 - 4bd \neq 0}} \hat{\sN}_q^r(x) \hat{f}_{\pm, 0}\left(\frac{\lambda^2}{q^2}\left(t^2 b, c + 2ub, \frac{d + cu + bu^2}{t^2}\right) \right) \frac{ds}{s^2}\\
  &=\sum_{q < Q} \mu(q)\oint_{\RE(s) = 4} X^s \int_1^\infty \frac{d\lambda}{\lambda} \int_0^\infty \frac{dt}{t} \int_{-\infty}^\infty du \lambda^{6-8s}t^{4s-1+z}\\&\times \sum_{\substack{x = (b,c,d)\\0 \leq c < 2|b| \\ c^2-4bd \neq 0}}\frac{1}{b} \hat{\sN}_q^r(x) \hat{f}_{\pm, 0}\left(\frac{\lambda^2}{q^2} \left(t^2 b, 2u, \frac{bd- \frac{c^2}{4} + u^2}{t^2b}\right)\right)\frac{ds}{s^2}. 
 \end{align*}
Write $q' = q/(q, 2^\infty)$ for the odd part of $q$. We now examine $\hat{\sN}_p^r(x)$ for odd $p | q$. Let $q_1$ be composed of those odd primes $p$ for which $p^2|b,c$ and $p|d$.  For these primes, $\hat{\sN}_p^r(x) = O\left(\frac{1}{p}\right)$.  Next let $q_2$ be composed of those odd primes $p$ for which $p^2$ does not divide both $b, c$ but $p|b,c,d$.  For these primes, $\hat{\sN}_p^r(x) = O\left(\frac{1}{p^3}\right)$.  Let $q_3$ be composed of those odd primes $p$ for which $p\nmid b$ but $p^2| c^2 -4bd$. For these primes $\hat{\sN}_p^r(x) = O\left(\frac{1}{p^3}\right)$.  Finally, let $q_4$ be composed of those odd primes $p$ for which $p\nmid b$ and $p\| c^2-4bd$.  For these $p$, $\hat{\sN}_p^r(x) = O\left(\frac{1}{p^4}\right).$ It follows that $|\hat{\sN}_q^r(x)| = O_\epsilon\left(\frac{q^\epsilon}{q_1q_2^3 q_3^3 q_4^4}\right)$ and $q_1^3 q_2^2 q_3^2 q_4 | c^2- 4bd$.  By the rapid decay of $\hat{f}$, $c^2 - 4bd = O(q^4 X^\epsilon)$ with negligible error.  For $d\ne 0$ and fixed $c\ne 0$, and a dyadic range $A \leq |c^2-4bd| < 2A$, $A \ll q^4 X^\epsilon$ there are $O\left(\frac{A X^\epsilon c^\epsilon}{q_1^3q_2^2q_3^2q_4} \right)$ ways of choosing $b,d$ to make the discriminant in the dyadic range.  We use $q_1^2 q_2 | c$ and $\frac{1}{c} > \frac{1}{2|b|}$ to obtain $$\sum_{\substack{b\ne 0, 0<c<2|b|,d\ne 0\\A\le|c^2-4bd|<2A\\q_1^2q_2|c\\q_1^3q_2^2q_3^2q_4|c^2-4bd}}|b|^{-1-2\epsilon}\ll_\epsilon\frac{A^{1+\epsilon/2}}{q_1^3q_2^2q_3^2q_4}\sum_{\substack{c\ge 1\\q_1^2q_2|c}}c^{-1-\epsilon}\ll_\epsilon\frac{A^{1+\epsilon/2}}{q_1^5q_2^3q_3^2q_4}$$as we sum over $c$. For $c=0$, the identity $4bd=-(c^2-4bd)$ gives $O_\epsilon(A^{\epsilon/2})$ choices for $b,d$. Since $|b|\ge q_1^2q_2$ this gives the same bound. For $d=0$, since $c>0, c^2\asymp A, q_1^2q_2q_3q_4|c$, $\sum_{\substack{|b|>c/2\\q_1^2q_2|b}}|b|^{-1-2\epsilon}\ll_\epsilon \frac{c^{-2\epsilon}}{q_1^2q_2}$ for each such $c$. Hence summing over $c$, this case contributes $$\ll_\epsilon\frac{A^{-1/2-\epsilon}}{q_1^4q_2^2q_3q_4}\ll_\epsilon\frac{A^{1+\epsilon/2}}{q_1^5q_2^3q_3^2q_4}.$$
Summing over the dyadic range, for sufficiently large $N$, $$
\sum_{A\ \mathrm{dyadic}}
A^{1+\epsilon/2}
\left(1+\frac{\lambda^4A}{q^4}\right)^{-N}
\ll_{\epsilon,N}
\left(\frac{q^4}{\lambda^4}\right)^{1+\epsilon/2}.
$$
Combining these estimates yields
\begin{align*}
&\left|
\hat Z_{q,r}^{+}\left(
\hat f_{\pm,0},2s+\frac12+\frac z2,
\frac54-3s-\frac z4
\right)\right|\ll_\epsilon q^{1+8\epsilon}.
\end{align*}
The estimates are locally uniform in $\RE s>1/4$. Shift the contour to $\RE(s) = \frac{1}{4}+\epsilon$. No pole is crossed, and the horizontal integrals vanish because of $s^{-2}$. We obtain
\begin{align*}
&\left|
\sum_{q<Q}\mu(q)\oint_{\RE s=4}
X^s
\hat Z_{q,r}^{+}\left(
\hat f_{\pm,0},2s+\frac12+\frac z2,
\frac54-3s-\frac z4
\right)\frac{ds}{s^2}
\right|\\
&\qquad\ll_\epsilon
X^{1/4+\epsilon}
\sum_{q<Q}q^{1+8\epsilon}
\ll_\epsilon
X^{1/4+\epsilon}Q^{2+8\epsilon}.
\end{align*}
\end{proof}

We now treat the contributions of the constant terms, handling the residues of the poles at the end.  From the main zeta function there are two constant term contributions that have poles, and only the poles survive (the remaining line integral can be pushed arbitrarily far left, saving an arbitrary power of $X$), but in the reducible there are two line integrals to bound aside from the residue of the pole.  These are bounded in the following lemma.

\begin{lemma}
 The difference between the integral
 \begin{align*}
  &\sum_{q < Q}\mu(q) \oint_{\RE(s) = 4}X^s \sum_{q_1q_2 = q} q_2^{-2} \prod_{p|q_1}p^{-(2s + \frac{1+z}{2})} \left(1 - \frac{1}{p^2}\right) \frac{\zeta\left(2s + \frac{1}{2} + \frac{z}{2}\right)}{s- \frac{5}{4} - \frac{z}{4}}\\&\times\left(\Phi_+ + \Phi_- \right)\left(f_{\pm,0}, 2s -\frac{1}{2} +\frac{z}{2}, 0\right) \frac{ds}{s^2}
 \end{align*}
and the residue of the pole at $\frac{5+z}{4}$ is $O(X^{\frac{1}{4} + \epsilon})$.
\end{lemma}
\begin{proof}
We first note that by Schwartz decay, the Mellin integral 
\begin{align*}
&(\Phi_++\Phi_-)
\left(f_{\pm,0},2s-\frac12+\frac z2,0\right)\\
&\qquad=
\int_{\mathbb R^3}
f_{\pm,0}(b,c,d)
|b|^{2s-\frac12+\frac z2}\,db\,dc\,dd.
\end{align*} 
is entire in $s$. Repeated integration by parts gives
\begin{align*}
&(\Phi_++\Phi_-)
\left(f_{\pm,0},2s-\frac12+\frac z2,0\right)\\
&\quad=
\frac{(-1)^N}{(2s+\frac12+\frac z2)^N}
\int_{\mathbb R^3}
(b\partial_b)^Nf_{\pm,0}(b,c,d)
|b|^{2s-\frac12+\frac z2}\,db\,dc\,dd,
\end{align*}
wich decreases faster than any negative power of $1+|\Im s|$ on any closed vertical strip.
 Shift the contour to $\RE(s) = \frac{1}{4} + \epsilon$, passing the contribution of the pole $s=\frac{5+z}{4}$. On this line, 
 $$
\left|\zeta\left(2s+\frac12+\frac z2\right)\right|
\le \zeta(1+2\epsilon),
$$
and
\begin{align*}
&\sum_{q<Q}|\mu(q)|
\sum_{q_1q_2=q}q_2^{-2}
\left|
\prod_{p\mid q_1}
p^{-(2s+\frac{1+z}{2})}(1-p^{-2})
\right|\\
&\qquad\le
\sum_{q_1,q_2\ge1}
q_1^{-1-2\epsilon}q_2^{-2}
=
\zeta(1+2\epsilon)\zeta(2).
\end{align*}
In particular, the arithmetic sum is bounded independently of $Q$. The residue is
\begin{align*}
&\frac{X^{(5+z)/4}}{((5+z)/4)^2}
\zeta(3+z)
(\Phi_++\Phi_-)(f_{\pm,0},2+z,0)\\
&\qquad\times
\sum_{q<Q}\mu(q)
\sum_{q_1q_2=q}
q_2^{-2}q_1^{-3-z}
\prod_{p\mid q_1}(1-p^{-2})\\
&=
\frac{X^{(5+z)/4}}{((5+z)/4)^2}
\zeta(3+z)
(\Phi_++\Phi_-)(f_{\pm,0},2+z,0)\\
&\qquad\times
\sum_{q<Q}\mu(q)
\prod_{p\mid q}
\left(p^{-2}+p^{-3-z}-p^{-5-z}\right).
\end{align*}
The rapid vertical decay of the $\Phi$-factor makes the horizontal integrals tend to zero. The remaining line integral is
$$
O_{\epsilon,f_{\pm,0},z}\left(X^{1/4+\epsilon}\right),
$$
  uniformly in $Q$.
\end{proof}

The remaining non-polar term is bounded as follows.
\begin{lemma}
 We have the bound
 \begin{align*}
  &\sum_{q <Q} \mu(q) \oint_{\RE(s) = 4} X^s \frac{\Sigma(\hat{f}_{\pm, 0}(2s -\frac{1}{2} + \frac{z}{2}))}{48s-20+4z}\\&\times\sum_{q_1q_2q_3 = \frac{q}{(q, 2^\infty)}} \prod_{p|q_1}\left(\frac{1}{p} + \frac{1}{p^2} - \frac{1}{p^3}\right) \prod_{p|q_2}\left(\frac{1}{p} - \frac{1}{p^3}\right) \prod_{p|q_3}\left(\frac{1}{p^3}-\frac{1}{p^4}\right)\\
  &\times \sum_{\substack{\ell \in \zed \setminus \{0\}\\ (\ell, q_2q_3) = 1}} \sum_{\substack{b = 1\\ (b,q_3)=1}}^\infty \frac{\omega_{2,q}(\ell, b)\phi(q_2b)q^{4s + 1 + z}}{(q_1^2q_2^3b^2 |\ell|)^{2s + \frac{1+z}{2}}} \frac{ds}{s^2}= O(X^{\frac{5}{12}+\epsilon}Q^{\frac{2}{3}+\epsilon}).
 \end{align*}

\end{lemma}
\begin{proof}
 We note that 
 $$
\Sigma\left(\hat f_{\pm,0},2s-\frac12+\frac z2\right)
\ll_N(1+|\operatorname{Im}s|)^{-N}
$$
on closed vertical strips in that half-plane. Next, since
$$
|\omega_{2,q}(\ell,b)|\le1,\qquad\phi(q_2b)\le q_2b,
$$
the sum over $\ell,b$, after omitting coprimality conditions, satisfies
$$
\sum_{\ell\ne0}\sum_{b\ge1}\frac{\phi(q_2b)}{b^{4\RE s+1}|\ell|^{2\RE s+\frac12}} \le 2q_2\zeta(4\RE s)\zeta\left(2\RE s+\frac12\right).
$$
This is finite for $\RE s>1/4$.
Also, we obtain 
\begin{align*}
\sum_{q<Q}|\mu(q)|\Bigg|&\sum_{q_1q_2q_3=q'}\prod_{p\mid q_1}(p^{-1}+p^{-2}-p^{-3}) \prod_{p\mid q_2}(p^{-1}-p^{-3})\\
&\times\prod_{p\mid q_3}(p^{-3}-p^{-4})\sum_{\substack{\ell\ne0\\(\ell,q_2q_3)=1}}\sum_{\substack{b\ge1\\(b,q_3)=1}}\frac{\omega_{2,q}(\ell,b)\phi(q_2b)q^{4s+1+z}}{(q_1^2q_2^3b^2|\ell|)^{2s+(1+z)/2}}\Bigg|\ll_{\sigma} Q^{4\sigma-1}.
\end{align*}
 
 Shift the contour to $\RE(s) = \frac{5}{12} + \epsilon$.  On this line the sum over $q$ is bounded by $Q^{\frac{2}{3}+\epsilon}$, which gives the bound claimed.
\end{proof}

It remains to evaluate the polar terms. By Proposition \ref{pole_proposition} and Lemma \ref{lemma:pole}, the poles at $\frac{5 \pm z}{4}$ cancel, leaving a pole at $\frac{11 \pm z}{12}$ to evaluate.  This is evaluated in the following lemma.

\begin{lemma}
   Let $P_{Q, \pm}$ be the contribution of the poles, sieved for $q \leq Q$. We have
\begin{align*}
P_{Q,-}={}&\frac13\tilde F\left(\frac{11+z}{12}\right)X^{\frac{11+z}{12}}2^{\frac{z-1}{6}}\pi^{\frac{1+2z}{6}}\cos\left(\frac{\pi(1-z)}6\right)\frac{\Gamma\left(\frac{1-z}{3}\right)\Gamma\left(\frac{2+z}{6}\right)\zeta\left(\frac{1-z}{3}\right)}{\Gamma\left(\frac{5+z}{6}\right)\zeta(2)\zeta\left(\frac{7+2z}{3}\right)}\\
&+\frac{\xi(z)}{\xi(1+z)}\frac13\tilde F\left(\frac{11-z}{12}\right)X^{\frac{11-z}{12}}2^{\frac{-z-1}{6}}\pi^{\frac{1-2z}{6}}\cos\left(\frac{\pi(1+z)}6\right)\frac{\Gamma\left(\frac{1+z}{3}\right)\Gamma\left(\frac{2-z}{6}\right)\zeta\left(\frac{1+z}{3}\right)}{\Gamma\left(\frac{5-z}{6}\right)\zeta(2)\zeta\left(\frac{7-2z}{3}\right)}\\
&+O_{F,z}\left(X^{11/12}Q^{-1}\right),
\end{align*}

\begin{align*}
P_{Q,+}={}&3^{-\frac{7+z}{4}}\tilde F\left(\frac{11+z}{12}\right)X^{\frac{11+z}{12}}2^{\frac{z-1}{6}}\pi^{\frac{1+2z}{6}}\cos\left(\frac{\pi(1-z)}6\right)\frac{\Gamma\left(\frac{1-z}{3}\right)\Gamma\left(\frac{2+z}{6}\right)\zeta\left(\frac{1-z}{3}\right)}{\Gamma\left(\frac{5+z}{6}\right)\zeta(2)\zeta\left(\frac{7+2z}{3}\right)}\\
&+\frac{\xi(z)}{\xi(1+z)}3^{-\frac{7-z}{4}}\tilde F\left(\frac{11-z}{12}\right)X^{\frac{11-z}{12}}2^{\frac{-z-1}{6}}\pi^{\frac{1-2z}{6}}\cos\left(\frac{\pi(1+z)}6\right)\frac{\Gamma\left(\frac{1+z}{3}\right)\Gamma\left(\frac{2-z}{6}\right)\zeta\left(\frac{1+z}{3}\right)}{\Gamma\left(\frac{5-z}{6}\right)\zeta(2)\zeta\left(\frac{7-2z}{3}\right)}\\
&+O_{F,z}\left(X^{11/12}Q^{-1}\right).
\end{align*}

\end{lemma}
\begin{proof}
 By the Lemma \ref{Theta_1_c_lemma}, the contribution from $s=(11+z)/12$, for negative discriminant, with the normalization factor $\frac{12}{\sqrt\pi K_{z/2}(2)}\frac{X^{(11+z)/12}}{((11+z)/12)^2}$, is
\begin{align*}
&\frac13\tilde F\left(\frac{11+z}{12}\right)X^{\frac{11+z}{12}} 2^{\frac{z-1}{6}}\pi^{\frac{1+2z}{6}}\cos\left(\frac{\pi(1-z)}6\right)\frac{\Gamma\left(\frac{1-z}{3}\right)\Gamma\left(\frac{2+z}{6}\right)\zeta\left(\frac{1-z}{3}\right)}{\Gamma\left(\frac{5+z}{6}\right)}\\
&\quad\times\sum_{q\le Q}\mu(q)\sum_{q_1q_2=q}q_2^{-2}q_1^{\frac{2(1-z)}3}\prod_{p\mid q_1}(p^{-3}-p^{-5})\\
&=\frac13\tilde F\left(\frac{11+z}{12}\right)X^{\frac{11+z}{12}}2^{\frac{z-1}{6}}\pi^{\frac{1+2z}{6}}\cos\left(\frac{\pi(1-z)}6\right)\frac{\Gamma\left(\frac{1-z}{3}\right)\Gamma\left(\frac{2+z}{6}\right)\zeta\left(\frac{1-z}{3}\right)}{\Gamma\left(\frac{5+z}{6}\right)}\\
&\quad\times \sum_{q\le Q}\mu(q)\prod_{p|q}\left(p^{-2}+p^{-\frac{7+2z}{3}}-p^{-\frac{13+2z}{3}}\right).
\end{align*}
Note that $$\sum_{q\ge 1}\mu(q)\prod_{p|q}\left(p^{-2}+p^{-\frac{7+2z}{3}}-p^{-\frac{13+2z}{3}}\right)=\frac{1}{\zeta(2)\zeta((7+2z)/3)}.$$

To estimate the tail of the $q$-sum, we obtain 
\begin{align*}
&\sum_{q>Q}|\mu(q)|\prod_{p\mid q}\left|p^{-2}+p^{-\frac{7+2z}{3}}-p^{-\frac{13+2z}{3}}\right|=\sum_{q>Q}|\mu (q)| q^{-2}\sum_{d\mid q}d^{-1/3}\\
&\le \sum_{q>Q}q^{-2}\sum_{d|q}d^{-1/3}\ll Q^{-1}.
\end{align*}
Since the polar contribution has absolute size $O_{F,z}(X^{11/12})$, replacing the truncated $q$-sum in (2) by the complete Euler product introduces an error
$$
O_{F,z}(X^{11/12}Q^{-1}).
$$
 We can obtain the similar result for positive discriminant.
\end{proof}

The proof of Theorem \ref{main_theorem} is the result of balancing the various error terms
$$
 O\left(X^{\frac{3}{4}+\epsilon} + \frac{X^{\frac{13}{12}}}{Q^{1-\epsilon}} + Q^{4+\epsilon}X^\epsilon + X^{\frac{1}{4}+\epsilon} Q^{2+\epsilon} + X^{\frac{5}{12} + \epsilon}Q^{\frac{2}{3}+\epsilon} + \frac{X^{\frac{11}{12}}}{Q}\right).
$$
The dominant terms are $\frac{X^{\frac{13}{12}+\epsilon}}{Q}$ and $Q^{4+\epsilon}$ which is optimal with $Q = X^{\frac{13}{60}}$ giving the claimed bound $O(X^{\frac{13}{15}+\epsilon})$.

\bibliographystyle{plain}

\begin{thebibliography}{1}
\bibitem{B53}
Erdilyi, Arthur, et al. 
\newblock \emph{Higher transcendental functions, vol. 1.} Bateman Manuscript Project, McGraw-Hill, New York (1953).

	
	
	
\bibitem{BH16} Bhargava, Manjul and Harron, Piper.
\newblock ``The equidistribution of lattice shapes of rings of integers in cubic, quartic, and quintic number fields.''
\newblock \emph{Compos. Math.}  152  (2016),  no. 6, 1111–1120. 

\bibitem{BST13}
Bhargava, Manjul,  Arul Shankar, and Jacob Tsimerman. 
\newblock ``On the Davenport-Heilbronn theorems and second order terms.''
\newblock \emph{Invent. Math.} 193 (2013), no. 2, 439--499.


\bibitem{DH71}
Davenport, Harold, and Hans Heilbronn. 
\newblock ``On the density of discriminants of cubic fields. II." 
\newblock \emph{Proceedings of the Royal Society of London. Series A, Mathematical and Physical Sciences} (1971): 405-420.

\bibitem{DF64} B.N. Delone and D.K. Faddeev.
\newblock ``The theory of irrationalities of the third degree.''
\newblock \emph{Translations of Mathematical Monographs} 10, A.M.S., Providence, RI, 1964.


\bibitem{GGS02}
W.-T. Gan, B. H. Gross, and G. Savin.
\newblock ``Fourier coefficients of modular forms on $G_2$.''
\newblock \emph{Duke Math. J.} 115 (2002), no. 1, pp. 105--169.


\bibitem{H17}
Hough, Bob. 
\newblock ``Maass form twisted Shintani $\mathcal{L}$-functions.''
\newblock \emph{Proc. Amer. Math. Soc.} 145 (2017), 4161-4174.

\bibitem{H19}
Hough, Robert. 
\newblock ``The shape of cubic fields." 
\newblock \emph{Research in the Mathematical Sciences} 6.3 (2019): 23.

\bibitem{HL20}
Hough, Robert and Eun Hye Lee.
\newblock ``Eisenstein series twisted Shintani zeta functions.'' 
\newblock \emph{Int. J. Number Theory} 22 (2026), 725-762.


\bibitem{I02}
Iwaniec, Henryk. 
\newblock \emph{Spectral methods of automorphic forms.}
\newblock Second edition. Graduate Studies in Mathematics, 53. American Mathematical Society, Providence, RI; Revista Matematica Iberoamericana, Madrid, 2002.

\bibitem{LTB22}
Lee, Eun Hye; Takloo-Bighash, Ramin.
\newblock ``On a multiple Dirichlet series associated to binary cubic forms.'' 
\newblock \emph{J. Number Theory} 238 (2022), 535-556


\bibitem{SS74}
Sato, Mikio; Shintani, Takuro. 
\newblock ``On zeta functions associated with prehomogeneous vector spaces.''
\newblock \emph{Ann. of Math.} (2) 100 (1974), 131--170. 


\bibitem{S72}
Shintani, Takuro.
\newblock ``On Dirichlet series whose coefficients are class numbers of integral binary cubic forms.''
\newblock \emph{J. Math. Soc. Japan}  24  1972 132–188.

\bibitem{S75}
Shintani, Takuro. 
\newblock ``On zeta functions associated with the vector space of quadratic forms." 
\newblock \emph{J. Fac. Sci. Univ. Tokyo} 22 (1975): 26-65.


\bibitem{T97}
Terr, David.
\newblock ``The distribution of shapes of cubic orders.'' PhD thesis, University of California,
Berkeley (1997).

\bibitem{TT13a}
Taniguchi, Takashi and Thorne, Frank.
\newblock ``Orbital $L$-functions for the space of binary cubic forms.''
\newblock \emph{Canad. J. Math.}  65  (2013), no. 6, 1320--1383. 	

\bibitem{TT13b}
Taniguchi, Takashi and Thorne, Frank.
\newblock ``Secondary terms in counting functions for cubic fields.''
\newblock \emph{Duke Math. J.}  162  (2013),  no. 13, 2451–2508.

\end{thebibliography}

\end{document}